\documentclass[3p]{elsarticle}

\usepackage{bm}
\usepackage[parfill]{parskip}
\usepackage[utf8]{inputenc}
\usepackage{url}
\usepackage{amsmath,amssymb,amsfonts,amsthm}
\usepackage{color}
\usepackage{bm}
\usepackage{graphicx}
\usepackage{subcaption}
\usepackage{lipsum}
\usepackage{amsfonts}
\usepackage{graphicx}
\usepackage{epstopdf}
\usepackage{array}
\usepackage[hidelinks]{hyperref}
\usepackage{lineno}
\usepackage{bbm}
\usepackage{geometry}

\usepackage{amsmath}
\usepackage{algorithm}
\usepackage[noend]{algpseudocode}

\ifpdf
  \DeclareGraphicsExtensions{.eps,.pdf,.png,.jpg}
\else
  \DeclareGraphicsExtensions{.eps}
\fi

\newtheorem{theorem}{Theorem}

\DeclareMathOperator*{\argmin}{arg\,min}

\makeatletter
\newsavebox\myboxA
\newsavebox\myboxB
\newlength\mylenA

\newcommand*\xoverline[2][0.75]{%
    \sbox{\myboxA}{$\m@th#2$}%
    \setbox\myboxB\null
    \ht\myboxB=\ht\myboxA%
    \dp\myboxB=\dp\myboxA%
    \wd\myboxB=#1\wd\myboxA
    \sbox\myboxB{$\m@th\overline{\copy\myboxB}$}
    \setlength\mylenA{\the\wd\myboxA}
    \addtolength\mylenA{-\the\wd\myboxB}%
    \ifdim\wd\myboxB<\wd\myboxA%
       \rlap{\hskip 0.5\mylenA\usebox\myboxB}{\usebox\myboxA}%
    \else
        \hskip -0.5\mylenA\rlap{\usebox\myboxA}{\hskip 0.5\mylenA\usebox\myboxB}%
    \fi}
\makeatother

\journal{arXiv.org}

\newcommand{\TheTitle}{Dynamical low-rank approximation for Burgers' equation with uncertainty} 

\date{\today}

\usepackage{amsopn}
\DeclareMathOperator{\diag}{diag}

\newcommand{\RomanNumeralCaps}[1]
    {\MakeUppercase{\romannumeral #1}}

\begin{document}
\begin{frontmatter}

\title{\TheTitle}

\author[adressJonas]{Jonas Kusch}
\author[adressGianluca]{Gianluca Ceruti}
\author[adressLukas]{Lukas Einkemmer}
\author[adressMartin]{Martin Frank}

\address[adressJonas]{Karlsruhe Institute of Technology, Karlsruhe, Germany,
    jonas.kusch@kit.edu}
\address[adressGianluca]{Universit{\"a}t T{\"u}bingen, T{\"u}bingen, Germany, ceruti@na.uni-tuebingen.de}
    \address[adressLukas]{University of Innsbruck, Innsbruck, Austria, lukas.einkemmer@uibk.ac.at}
\address[adressMartin]{Karlsruhe Institute of Technology, Karlsruhe, Germany, martin.frank@kit.edu}

\begin{abstract}
Quantifying uncertainties in hyperbolic equations is a source of several challenges. First, the solution forms shocks leading to oscillatory behaviour in the numerical approximation of the solution. Second, the number of unknowns required for an effective discretization of the solution grows exponentially with the dimension of the uncertainties, yielding high computational costs and large memory requirements. An
efficient representation of the solution via adequate basis functions permits to tackle these difficulties. The generalized polynomial chaos (gPC) polynomials allow such an efficient representation when the distribution of the uncertainties is known. These distributions are usually only available for input uncertainties such as initial conditions, therefore the efficiency of this ansatz can get lost during runtime.
In this paper, we make use of the dynamical low-rank approximation (DLRA) to obtain a memory-wise efficient solution approximation on a lower dimensional manifold. We investigate the use of the matrix projector-splitting integrator and the unconventional integrator for dynamical low-rank approximation, deriving separate time evolution equations
for the spatial and uncertain basis functions, respectively. This guarantees an efficient approximation of the solution even if the underlying probability distributions change over time. Furthermore, filters to mitigate the appearance of spurious oscillations are implemented, and a strategy to enforce boundary conditions is introduced. The proposed methodology is analyzed for Burgers’ equation equipped with uncertain initial
values represented by a two-dimensional random vector. The numerical experiments validate that the results
of a standard filtered Stochastic-Galerkin (SG) method are consistent with the numerical results obtained via
the use of numerical integrators for dynamical low-rank approximation. Significant reduction of the memory
requirements is obtained, and the important characteristics of the original system are well captured.
\end{abstract}


\end{frontmatter}

\section{Introduction}
A vast amount of engineering applications such as hydrology, gas dynamics, or radiative transport are governed by hyperbolic conservation laws. In many applications of interest the inputs (e.g. initial-, boundary conditions, or modeling parameters) of these equations are uncertain.  These uncertainties arise from modeling assumptions as well as measurement or discretization errors, and they can heavily affect the behavior
of inspected systems. Therefore, one core ingredient to obtain reliable knowledge of a given application is
to derive methods which include the effects of uncertainties in the numerical simulations.

Methods to quantify effects of uncertainties can be divided into intrusive and non-intrusive methods. Non-intrusive methods run a given deterministic solver for different realizations of the input in a black-box manner, see e.g. \cite{xiu2005high,babuvska2007stochastic,loeven2008probabilistic} for the stochastic-Collocation method and \cite{heinrich2001multilevel,mishra2012multi,mishra2012sparse,mishra2016numerical} for (multi-level) Monte Carlo methods. Intrusive methods perform a modal discretization of the solution and derive time evolution equations for the corresponding expansion coefficients. The perhaps most prominent intrusive approach is the stochastic-Galerkin (SG) method \cite{ghanem2003stochastic}, which represents the random dimension with the help of polynomials. These polynomials, which are picked according to a chosen probability distribution, are called generalized polynomial chaos (gPC) functions \cite{wiener1938homogeneous,xiu2002wiener}. By performing a Galerkin projection, a set of deterministic evolution equations for the expansion coefficients called the moment system can be derived. 

Quantifying uncertainties in hyperbolic problems comes with a large number of challenges such as spurious oscillations \cite{le2004uncertainty,barth2013non,dwight2013adaptive} or the loss of hyperbolicity \cite{poette2009uncertainty}. A detailed discussion and numerical comparison of these challenges when using intrusive and non-intrusive methods can be found in \cite{kusch2020intrusive}. Intrusive methods which preserve hyperbolicity are the intrusive polynomial moment (IPM) method \cite{poette2009uncertainty} which performs a gPC expansion of the entropy variables, and the Roe transformation method \cite{pettersson2014stochastic,gerster2020entropies} which performs a gPC expansion of the Roe variables. Furthermore, admissibility of the solution can be achieved by using bound-preserving limiters \cite{schlachter2018hyperbolicity} to push the solution into an admissible set. Oscillations that frequently arise in hyperbolic problems can be mitigated by either filters \cite{kusch2020filtered} or the use of multi-elements \cite{wan2006multi,durrwachter2020hyperbolicity,kusch2020oscillation}. 

One key challenge of uncertainty quantification is the exponential growth of the number of unknowns when the dimension of the random domain increases. Hyperbolic conservation laws, which tend to form shocks amplify this effect, since they require a fine discretization in each dimension. This does not only yield higher computational costs, but also extends memory requirements. Therefore, a crucial task is to find an efficient representation which yields a small error when using a small number of expansion coefficients. The gPC expansion provides such an efficient representation if the chosen polynomials belong to the probability density of the solution \cite{xiu2002wiener}. However, the probability density is commonly known only for the initial condition and the time evolution of this density is not captured by the chosen gPC polynomials. Thus, choosing gPC polynomials according to the initial distribution can become inaccurate when the distribution changes over time. 

We aim to resolve these issues by applying the dynamical low-rank approximation (DLRA) \cite{koch2007dynamical} to our problem. To decrease memory and computational requirements DLRA represents and propagates the solution in time on a prescribed low-rank manifold. Such a low-rank representation is expected to be efficient, since choosing the gPC expansion according to the underlying probability density provides an accurate solution representation for a small number of expansion coefficients \cite{xiu2002wiener}. Solving the DLRA equation by a matrix projector-splitting integrator \cite{lubich2014projector} yields time evolution equations for updating spatial and uncertain basis functions in time. Hence, the resulting method is able to automatically adapt basis function to resolve important characteristics of the inspected problem. Furthermore, the matrix projector-splitting integrator has improved stability properties and error bounds \cite{kieri2016discretized}. An extension of the matrix projector-splitting integrator to function spaces is presented in \cite{einkemmer2018low}.
First applications of dynamical low-rank approximations to uncertainty quantification are \cite{sapsis2009dynamically,sapsis2012dynamical,ueckermann2013numerical} which use a so-called \textit{dynamical double orthogonal} (DDO) approximation. The method however depends on the regularity of the coefficient matrix, which can potentially restrict the time step. Applications of the DLRA method in combination with the matrix projector-splitting integrator for parabolic equations with uncertainty can for example be found in \cite{musharbash2015error,musharbash2018dual,kazashi2020stability,kazashi2020existence}. A dynamical low-rank approximation for random wave equations has been studied in \cite{musharbash2017symplectic}. Examples of DLRA for kinetic equations, which include hyperbolic advection terms are \cite{einkemmer2018low,einkemmer2019quasi,einkemmer2020low,einkemmer2021asymptotic,peng2019low,peng2020high}. Similar to our work, filters have been used in \cite{peng2019low} to mitigate oscillations in the low-rank approximation. Furthermore, \cite{einkemmer2020low} uses diffusion terms to dampen oscillatory artifacts.

In this work, we focus on efficiently applying the matrix projector-splitting integrator and the unconventional integrator for DLRA to hyperbolic problems with quadratic physical fluxes. Furthermore, we study the effects of oscillatory solution artifacts, which we aim to mitigate through filters. Furthermore, we investigate a strategy to preserve boundary conditions, similar to \cite{musharbash2018dual}. Additionally, we investigate the unconventional DLRA integrator \cite{ceruti2020unconventional} in the context of uncertainty quantification and compare it to the standard matrix projector-splitting integrator \cite{lubich2014projector}. The different integrators for dynamical low-rank are compared to the classical stochastic-Galerkin method for Burgers' equation with two-dimensional uncertainties. In our numerical experiments, the chosen methods for DLRA capture the highly-resolved stochastic-Galerkin results nicely. It is observed that the unconventional integrator smears out the solution, which improves the expected value of the approximation but yields heavy dampening for the variance.

Following the introduction, we briefly present the required background for this paper in Section~\ref{sec:background}. Here, we give an overview of intrusive UQ methods as well as the DLRA framework applied to uncertainty quantification. Section~\ref{sec:ProjectorSplitting} discusses the matrix projector-splitting integrator applied to a scalar, hyperbolic equation with uncertainty. Section~\ref{sec:BC} proposes a strategy to enforce Dirichlet boundary conditions in the low-rank numerical approximation and Section~\ref{sec:discretize} discusses the numerical discretization. In Section~\ref{sec:numResults} we demonstrate the effectiveness of the dynamical low-rank approximation ansatz for hyperbolic problems by investigating Burgers' equation with a two-dimensional uncertainty.

\section{Background}\label{sec:background}
We compute a low-rank solution of a scalar hyperbolic equation with uncertainty
\begin{subequations}\label{eq:hyperbolicProblem}
\begin{align}
\partial_t &u(t,x,\xi) + \partial_x f(u(t,x,\xi)) = 0, \\ 
&u(t=0,x,\xi) = u_{\text{IC}}(x,\xi),\label{eq:ic} \\
u(t,x_L,\xi) &= u_L(t,\xi) \enskip \text{ and } \enskip u(t,x_R,\xi) = u_R(t,\xi).
\end{align}
\end{subequations}
The solution $u$ depends on time $t\in\mathbb{R}_+$, space $x\in [x_L,x_R]\subset\mathbb{R}$ and a scalar random variable $\xi\in\Theta\subset\mathbb{R}$. The random variable $\xi$ is equipped with a known probability density function $f_{\Xi}:\Theta\rightarrow\mathbb{R}_+$. 
\subsection{Intrusive methods for uncertainty quantification}\label{sec:backgroundSG}
The core idea of most intrusive methods is to represent the solution to \eqref{eq:hyperbolicProblem} by a truncated gPC expansion
\begin{align}\label{eq:truncatedGPC}
u(t,x,\xi)\approx u_N(t,x,\xi):=\sum_{i=0}^N \widehat u_i(t,x)\varphi_i(\xi) = \bm{\widehat{u}}(t,x)^T\bm{\varphi}(\xi).
\end{align}
Here, the basis functions $\bm{\varphi}=(\varphi_0,\cdots,\varphi_N)^T$ are chosen to be orthonormal with respect to the probability density function $f_{\Xi}$, i.e.,
\begin{align*}
\mathbb{E}[\varphi_i\varphi_j] = \int_{\Theta}\varphi_i(\xi)\varphi_j(\xi)f_{\Xi}\,d\xi = \delta_{ij}.
\end{align*}
Note that the chosen polynomial ansatz yields an efficient evaluation of quantities of interest such as expected value and variance
\begin{align*}
\mathbb{E}[u_N(t,x,\cdot)] = \widehat u_0(t,x), \qquad \text{Var}[u_N(t,x,\cdot)] = \sum_{i=1}^N \widehat u_i(t,x)^2.
\end{align*}
This ansatz is used to represent the solution in \eqref{eq:hyperbolicProblem} which yields
\begin{align}\label{eq:hyperbolicTruncated}
\partial_t u_N(t,x,\xi) + \partial_x f(u_N(t,x,\xi)) = R(t,x,\xi).
\end{align}
Then, a Galerkin projection is performed. We project the resulting residual $R$ to zero by multiplying \eqref{eq:hyperbolicTruncated} with test functions $\varphi_i$ (for $i=0,\cdots,N$) and taking the expected value. Prescribing the residual to be orthogonal to the test functions gives
\begin{align*}
\partial_t \widehat u_i(t,x) + \partial_x \mathbb{E}\left[f(u_N(t,x,\cdot))\varphi_i\right] = 0\qquad\text{ for }i = 0,\cdots,N.
\end{align*}
This closed deterministic system is called the stochastic-Galerkin (SG) moment system. It can be solved with standard finite volume or discontinuous Galerkin methods, provided it is hyperbolic. The resulting moments can then be used to evaluate the expected solution and its variance. However, hyperbolicity is not guaranteed for non-scalar problems, which is why a generalization of SG has been proposed in \cite{poette2009uncertainty}. This generalization, which is called the intrusive polynomial moment (IPM) method performs the gPC expansion on the so-called entropy variable $v = s'(u)$, where $s:\mathbb{R}\rightarrow\mathbb{R}$ is a convex entropy to \eqref{eq:hyperbolicProblem}. For more details on the IPM method, we refer to the original IPM paper \cite{poette2009uncertainty} as well as \cite[Chapter~4.2.3]{poette2019contribution} and \cite[Chapter~1.4.5]{10.5445/IR/1000121168}.
Furthermore, the solutions of various UQ methods, including stochastic-Galerkin, show spurious solution artifacts such as non-physical step wise approximations \cite{le2004uncertainty,barth2013non,dwight2013adaptive}. One strategty to mitigate these effects are filters, which have been proposed in \cite{kusch2020filtered}. The idea of filters is to dampen high order moments in between time steps to ensure a smooth approximation in the uncertain domain $\Theta$. 

When denoting the moments evaluated in a spatial cell $j$ at time $t_n$ as $\bm{\widehat u}_j^n\in\mathbb{R}^{N+1}$, a finite volume update with numerical flux $\bm{F^*}:\mathbb{R}^{N+1}\times\mathbb{R}^{N+1}\rightarrow\mathbb{R}^{N+1}$ takes the form
\begin{align}\label{eq:FVEquations}
 \bm{\widehat{u}}_j^{n+1} = \bm{\widehat{u}}_j^{n} -
 \frac{\Delta t}{\Delta x}(\bm{F^*}(\bm{\widehat{u}}_{j}^n,
 \bm{\widehat{u}}_{j+1}^n) - \bm{F^*}(\bm{\widehat{u}}_{j-1}^n,
 \bm{\widehat{u}}_{j}^n)).
\end{align}
To reduce computational costs, one commonly chooses a kinetic flux 
\begin{align*}
  \bm{F^*}(\bm{\widehat{u}},\bm{\widehat{v}}) = \mathbb{E}\left[ f^*(\bm{\widehat{u}}^T\bm{\varphi},\bm{\widehat{v}}^T\bm{\varphi})\bm{\varphi} \right]\approx\sum_{k=1}^{N_q}w_k f^*(\bm{\widehat{u}}^T\bm{\varphi}(\xi_k),\bm{\widehat{v}}^T\bm{\varphi}(\xi_k))\bm{\varphi}(\xi_k),
\end{align*}
where $f^*$ is a numerical flux for the original problem \eqref{eq:hyperbolicProblem} and we make use of a quadrature rule with $N_q$ points $\xi_k$ and weights $w_k$. An efficient computation is achieved by precomputing and storing the $N\cdot N_q$ terms $\varphi_i(\xi_k)$. For a more detailed derivation of the finite volume scheme, see e.g. \cite{kusch2020filtered}. Now, to mitigate spurious oscillations in the uncertain domain, we derive a filtering matrix $\bm{\mathcal{F}}\in\mathbb{R}^{N+1\times N+1}$. Note that to approximate a given function $u_{ex}:\Theta\rightarrow\mathbb{R}$, the common gPC expansion minimizes the L$^2$-distance between the polynomial approximation and $u_{ex}$. Unfortunately, this approximation tends to oscillate. A representation
\begin{align*}
    p(\xi) = \sum_{i=0}^N \widehat{\alpha}_i \varphi_i(\xi) = \bm{\widehat{\alpha}}^T\bm{\varphi}(\xi)
\end{align*}
which dampens oscillations in the polynomial approximation can be derived by solving the optimization problem
\begin{align}\label{eq:filterOptProblem}
   \bm{\widehat{\alpha}} = \argmin_{\bm{\alpha}\in\mathbb{R}^{N+1}}\left\{\mathbb{E}\left[ \left(\bm{\alpha}^T\bm{\varphi} - u_{ex}\right)^2 \right] + \lambda \mathbb{E}\left[\left( \mathcal{L}\bm{\alpha}^T\bm{\varphi}\right)^2\right]\right\}.
\end{align}
Here, $\lambda\in\mathbb{R}_+$ is a user-determined filter strength and $\mathcal{L}$ is an operator which returns high values when the approximation oscillates. Note that we added a term which punishes oscillations in the L$^2$-distance minimization that is used in gPC expansions. For uniform distributions, a common choice of $\mathcal{L}$ is
\begin{align*}
    \mathcal{L} := \frac{d}{d\xi}\left(1-\xi^2\right)\frac{d}{d\xi}.
\end{align*}
In this case, the gPC polynomials are eigenfunctions of $\mathcal{L}$.
In this case, the optimal expansion coefficients $\bm{\widehat \alpha}$ are given by $\widehat{\alpha}_i= \frac{1}{1+\lambda i^2 (i-1)^2}\mathbb{E}[u_{ex}\varphi_i]$, i.e., high order moments of the function $u_{ex}$ will be dampened. Collecting the dampening factors in a matrix $\bm{\mathcal{F}}(\lambda) = \diag\{(1 + \lambda i^2(i + 1)^2)^{-1}\}_{i = 0}^{N}$ and applying this dampening steps in between finite volume updates yields the filtered SG method
\begin{subequations}\label{eq:filteredEquations}
\begin{align}
 \bm{\overline{u}}_j^n &= \bm{\mathcal{F}}(\lambda) \bm{\widehat u}_j^n, \\
 \bm{\widehat u}_j^{n+1} &= \bm{\overline{u}}_j^{n} -
 \frac{\Delta t}{\Delta x}(\bm{F^*}(\bm{\overline{u}}_{j}^n,
 \bm{\overline{u}}_{j+1}^n) - \bm{F^*}(\bm{\overline{u}}_{j-1}^n,
 \bm{\overline{u}}_{j}^n)).
\end{align}
\end{subequations}
A filter for high dimensional uncertainties can be applied by successively adding a punishing term to the optimization problem \eqref{eq:filterOptProblem} for every random dimension. When the random domain is two-dimensional, i.e., $\Theta\subset\mathbb{R}^2$, then with
\begin{align*}
    \mathcal{L}_1 := \frac{\partial}{\partial\xi_1}\left(1-\xi_1^2\right)\frac{\partial}{\partial\xi_1}, \qquad \mathcal{L}_2 := \frac{\partial}{\partial\xi_2}\left(1-\xi_2^2\right)\frac{\partial}{\partial\xi_2},
\end{align*}
one can use
\begin{align*}
    \bm{\widehat{\alpha}} = \argmin_{\bm{\alpha}\in\mathbb{R}^{N+1}}\left\{\mathbb{E}\left[ \left(\bm{\alpha}^T\bm{\varphi} - u_{ex}\right)^2 \right] + \lambda \mathbb{E}\left[\left( \mathcal{L}_1\bm{\alpha}^T\bm{\varphi}\right)^2\right]+ \lambda \mathbb{E}\left[\left( \mathcal{L}_2\bm{\alpha}^T\bm{\varphi}\right)^2\right]\right\}.
\end{align*}
Here, the gPC functions are tensorized and collected in a vector $\varphi_{i(N+1)+j}(\xi_1,\xi_2) = \varphi_i(\xi_1)\varphi_j(\xi_2)$ where $i,j=0,\cdots,N$. The resulting filtered expansion coefficients are then given by 
\begin{align*}
    \widehat{\alpha}_{i(N+1)+j}= \frac{1}{1+\lambda i^2 (i-1)^2+\lambda j^2 (j-1)^2}\mathbb{E}[u_{ex}\varphi].
\end{align*}
The costs of the filtered SG as well as the classical SG method when using a kinetic flux function are $C_{SG}\lesssim N_t\cdot N_x\cdot N\cdot N_q$ and the memory requirement is $M_{SG} \lesssim \max\{ N_x\cdot N, N\cdot N_q \}$. For more general problems with uncertain dimension $p$, tensorizing the chosen basis functions and quadrature yields $C_{SG}\lesssim N_t\cdot N_x\cdot N^p\cdot N_q^p$ and the memory requirement is $M_{SG} \lesssim \max\{ N_x\cdot N^p, N\cdot N_q \}$. For a discussion of the advantages when using a kinetic flux, see \cite[Appendix~A]{kusch2020intrusive}.

\subsection{Matrix projector-splitting integrator for dynamical low-rank approximation}
\label{sec:splittingIntegrator}
The core idea of the dynamical low-rank approach is to project the original problem on a prescribed manifold of rank $r$ functions. Such an approximation is given by
\begin{align}\label{eq:rankrsol}
u(t,x,\xi)\approx\sum_{i,\ell=1}^r X_{\ell}(t,x) S_{\ell i}(t) W_i(t,\xi).
\end{align}
In the following, we denote the set of functions, which have a representation of the form \eqref{eq:rankrsol} by $\mathcal{M}_r$. Then, instead of computing the best approximation in $\mathcal{M}_r$ the aim is to find a solution $u_r(t,x,\xi)\in\mathcal{M}_r$ fulfilling
\begin{align}\label{eq:DLRproblem}
\partial_t u_r(t,x,\xi)\in T_{ u_r(t,x,\xi)}\mathcal{M}_r \qquad \text{such that} \qquad \Vert \partial_t u_r(t,\cdot,\cdot)+\partial_x f(u(t,\cdot,\cdot))\Vert = \text{min}.
\end{align}
Here, $T_{u_r(t,x,\xi)}\mathcal{M}_r$ denotes the tangent space of $\mathcal{M}_r$ at $u_r(t,x,\xi)$. According to \cite{einkemmer2018low}, the orthogonal projection onto the tangent space reads
\begin{align*}
&Pg = P_X g - P_X P_W g + P_{W}g, \\
\text{where } &P_X g = \sum_{i=1}^r X_i \langle X_i, g\rangle, \quad P_{W}g = \sum_{j=1}^r W_j\mathbb{E}[ W_j g ].
\end{align*}
This lets us rewrite \eqref{eq:DLRproblem} as
\begin{align}\label{eq:lowRankProjector}
\partial_t u_r(t,x,\xi) = -P(u_r(t,x,\xi))\partial_x f(u(t,x,\xi)),
\end{align}
where $\langle\cdot\rangle$ denotes the integration over the spatial domain. A detailed derivation can be found in \cite[Lemma~4.1]{koch2007dynamical}. Then, a Lie-Trotter splitting technique yields
\begin{subequations}\label{eq:projectorSplitEq}
\begin{align}
\partial_t u_{\RomanNumeralCaps{1}} &= P_W\left( -\partial_x f(u_{\RomanNumeralCaps{1}}) \right), \label{eq:DLR1}\\ 
\partial_t u_{\RomanNumeralCaps{2}} &= P_W P_X \left( \partial_x f(u_{\RomanNumeralCaps{2}}) \right), \label{eq:DLR2}\\ 
\partial_t u_{\RomanNumeralCaps{3}} &= P_X \left( -\partial_x f(u_{\RomanNumeralCaps{3}})\right) \label{eq:DLR3}.
\end{align}
\end{subequations}
In these split equations, each solution has a decomposition of the form \eqref{eq:rankrsol}. The solution of these split equations can be further simplified: Let us write $u_{\RomanNumeralCaps{1}}(t,x,\xi) = \sum_{j=1}^{r}K_j(t,x)W_{j}(t,\xi)$, i.e., we define $K_j(t,x):=\sum_{i=1}^r X_i(t,x)S_{ij}(t)$. Then, we test \eqref{eq:DLR1} against $W_{n}$ and omit the index $\RomanNumeralCaps{1}$ in the decomposition to simplify notation. This gives
\begin{align}\label{eq:KstepGeneral}
\partial_t K_n(t,x) = -\mathbb{E}\left[\partial_{x}f\left(u_{\RomanNumeralCaps{1}}(t,x,\cdot)\right)W_n(t,\cdot)\right].
\end{align}
This system is similar to the SG moment system, but with time dependent basis functions, cf. \cite{tryoen2010intrusive}. Performing a Gram-Schmidt decomposition in $L^2$ of $K_n$ yields time updated $X_i$ and $S_{ij}$. Then testing \eqref{eq:DLR2} with $W_m$ and $X_k$ gives
\begin{align}\label{eq:SstepGeneral}
\dot{S}_{mk}(t) = \mathbb{E}\left[\left\langle \partial_x f(u_{\RomanNumeralCaps{2}}(t,\cdot,\cdot))X_k(t,\cdot)\right\rangle W_m(t,\cdot)\right].
\end{align}
This equation is used to update $S_{mk}$. Lastly, we write $u_{\RomanNumeralCaps{3}}(t,x,\xi) = \sum_{i=1}^{r}X_i(t,x)L_i(t,\xi)$ and test \eqref{eq:DLR3} with $X_{n}$. This then yields
\begin{align}\label{eq:LstepGeneral}
\partial_t L_n(t,\xi) = \left\langle -\partial_x f(u_{\RomanNumeralCaps{3}}(t,\cdot,\xi))X_n(t,\cdot)\right\rangle.
\end{align}
Again, a Gram-Schmidt decomposition is used on $L_n$ to determine the time updated quantities $W_i$ and $S_{ij}$. Note that the $K$-step does not modify the basis $W$ and the $L$-step does not modify the basis $X$. Furthermore, the $S$-step solely alters the coefficient matrix $S$. Therefore, the derived equations can be interpreted as time update equations for the spatial basis functions $X_{\ell}$, the uncertain basis functions $W_{i}$ and the expansion coefficients $S_{\ell i}$. Hence, the matrix projector-splitting integrator evolves the basis functions in time, such that a low-rank solution representation is maintained. Let us use Einstein's sum notation to obtain a compact presentation of the integrator. The projector splitting procedure which updates the basis functions $X^{0} := X(t_0,x)$, $W^{0} := W(t_0,\xi)$ and coefficients $S^0 = S(t_0)$ from time $t_0$ to $t_1 = t_0+\Delta t$ then takes the following form:
\begin{enumerate}
    \item \textbf{$K$-step}: Update $X^{0}$ to $X^{1}$ and $S^0$ to $\widehat S^1$ via
\begin{align*}
\partial_t K_n(t,x) &= -\mathbb{E}\left[\partial_{x}f\left(K_{\ell} W^0_{\ell}\right)W_n^0\right]\\
K_n(t_0,x) &= X_{i}^{0}S_{in}^0.
\end{align*}
Determine $X^1$ and $\widehat S^1$ with $K(t_1,x) = X^1 \widehat S^1$.
\item \textbf{$S$-step}: Update $\widehat S^1 \rightarrow \widetilde S^0$ via
\begin{align*}
\dot{S}_{mk}(t) &= \mathbb{E}\left[\left\langle \partial_x f(X_{\ell}^1\widehat S_{\ell i}^1 W^0_i))X_k^1(t,\cdot)\right\rangle W_m^0\right]\\
S_{mk}(t_0) &= \widehat S^1_{mk}
\end{align*}
and set $\widetilde S^0 = S(t_0+\Delta t)$.
\item \textbf{$L$-step}: Update $W^0 \rightarrow W^1$ and $\widetilde S^0\rightarrow S^1$ via
\begin{align*}
\partial_t L_n(t,\xi) &= \left\langle -\partial_x f(X_{\ell i}^1 L_i) X_n^1\right\rangle\\
L_n(t_0,\xi) &= \widetilde S_{ni}^0 W_i^0.
\end{align*}
Determine $W^1$ and $S^1$ with $L^1 = S^1 W^1$.
\end{enumerate}

\subsection{Unconventional integrator for dynamical low-rank approximation}
Note that the $S$-step \eqref{eq:SstepGeneral} evolves the matrix $S$ backward in time, which is a source of instability for non-reversible problems such as diffusion equations or particle transport with high scattering. Furthermore, the presented equation must be solved successively, which removes the possibility of solving different steps in parallel. In \cite{ceruti2020unconventional}, a new integrator which enables parallel treatment of the $K$ and $L$-step while only evolving the solution forward in time has been proposed. This integrator, which is called the unconventional integrator, is similar but not equal to the matrix projector-splitting integrator and works as follows:
\begin{enumerate}
    \item \textbf{$K$-step}: Update $X^{0}$ to $X^{1}$ via
\begin{align*}
\partial_t K_n(t,x) &= -\mathbb{E}\left[\partial_{x}f\left(K_{\ell} W^0_{\ell}\right)W_n^0\right]\\
K_n(t_0,x) &= X_{i}^{0}S_{in}^0.
\end{align*}
Determine $X^1$ with a QR-decomposition $K(t_1,x) = X^1 R$ and store $M = \left(\langle X^1_i X^0_j\rangle\right)_{i,j = 1}^r$.
\item \textbf{$L$-step}: Update $W^0$ to $W^1$ via
\begin{align*}
\partial_t L_n(t,\xi) &= \left\langle -\partial_x f(X_{i}^0 L_i) X_n^1\right\rangle\\
L_n(t_0,\xi) &= S_{ni}^0 W_i^0.
\end{align*}
Determine $W^1$ with a QR-decomposition $L^1 = W^1\widetilde R$ and store $N = \left(\mathbb{E}[W_i^1 W_j^0]\right)_{i,j = 1}^r$.
\item \textbf{$S$-step}: Update $S^0$ to $S^1$ via
\begin{align*}
\dot{S}_{mk}(t) &= -\mathbb{E}\left[\left\langle \partial_x f(X_{\ell}^1S_{\ell i} W^1_i))X_k^1(t,\cdot)\right\rangle W_m^1\right]\\
S_{mk}(t_0) &= M_{m\ell}S_{\ell j}^0 N_{kj}
\end{align*}
and set $S^1 = S(t_0+\Delta t)$.
\end{enumerate}
Note that the first and second steps can be performed in parallel. Furthermore, the unconventional integrator inherits the exactness and robustness properties of the classical matrix projector-splitting integrator, see \cite{ceruti2020unconventional}. Additionally, it allows for an efficient use of rank adaptivity \cite{CKL2021}.

\section{Matrix projector-splitting integrator for uncertain hyperbolic problems}\label{sec:ProjectorSplitting}
Before deriving the evolution equations for scalar problems, we first discuss hyperbolicity.
\begin{theorem}\label{th:hyperbolicityKStep}
Provided the original scalar equation \eqref{eq:hyperbolicProblem} is hyperbolic, its corresponding $K$-step equation \eqref{eq:KstepGeneral} is hyperbolic as well, i.e., its flux Jacobian is diagonalizable with real eigenvalues.
\end{theorem}
\begin{proof}
First, we apply the chain rule to the $K$-step equation to obtain
\begin{align*}
\partial_t K_n(t,x) = -\mathbb{E}\left[f'\left(u_{\RomanNumeralCaps{1}}(t,x,\cdot)\right)\sum_{j=1}^{r}\partial_x K_j(t,x)W_{j}(t,\cdot)W_n(t,\cdot)\right].
\end{align*}
Linearity of the expected value gives
\begin{align*}
\partial_t K_n(t,x) = -\sum_{j=1}^{r}\mathbb{E}\left[f'\left(u_{\RomanNumeralCaps{1}}(t,x,\cdot)\right)W_{j}(t,\cdot)W_n(t,\cdot)\right] \partial_x K_j(t,x).
\end{align*}
Then, the flux Jacobian is obviously symmetric, i.e., by the spectral theorem it is diagonalizable with real eigenvalues.
\end{proof}
We remind the reader that the $K$-step of the unconventional integrator equals the $K$-step of the matrix projector-splitting integrator. Therefore, Theorem~\ref{th:hyperbolicityKStep} holds for both numerical integrators. Note that this result only holds for scalar equations. In the system case, hyperbolicity cannot be guaranteed. Methods to efficiently guarantee hyperbolicity for systems are not within the scope of this paper and will be left for future work.

Now, we apply the matrix projector-splitting integrator presented in Section~\ref{sec:splittingIntegrator} to Burgers' equation
\begin{subequations}\label{eq:Burgers}
\begin{align}
\partial_t &u(t,x,\xi)+\partial_x \frac{(u(t,x,\xi))^2}{2} = 0, \\
&u(t=0,x,\xi) = u_{\text{IC}}(x,\xi),\\
u(t,x_L,\xi) &= u_L(t,\xi) \enskip \text{ and } \enskip u(t,x_R,\xi) = u_R(t,\xi).
\end{align}
\end{subequations}
Our goal is to never compute the full solutions $u_{\RomanNumeralCaps{1}}(t,x,\xi)$, $u_{\RomanNumeralCaps{2}}(t,x,\xi)$ and $u_{\RomanNumeralCaps{3}}(t,x,\xi)$ but to rather work on the decomposed quantities saving memory and reducing computational costs. Here, we represent the uncertain basis functions with gPC polynomials, i.e.,
\begin{align}\label{eq:basisW}
    W_i(t,\xi) = \sum_{m=0}^N \widehat{W}_{mi}(t)\varphi_m(\xi) \qquad \text{ for } i = 1,\cdots,r.
\end{align}
The spatial basis functions are represented with the help of finite volume (or DG0) basis functions 
\begin{align}\label{eq:DG0Basis}
   Z_\ell(x) = \frac{1}{\sqrt{\Delta x}}\chi_{[x_{\ell-1/2},x_{\ell+1/2}]}(x),
\end{align}
where $\chi_A$ is the indicator function on the interval $A\subset\mathbb{R}$. Then, the spatial basis functions are given by
\begin{align}
    X_i(t,x) = \sum_{m=1}^{N_x} \widehat{X}_{mi}(t)Z_m(x) \qquad \text{ for } i = 1,\cdots,r.
\end{align}
In the following, we present an efficient evaluation of non-linear terms that arise in the matrix projector-splitting integrator for the non-linear equations.
\subsection{$K$-step}
Let us start with the $K$-step, which decomposes the solution $u_{\RomanNumeralCaps{1}}$ into
\begin{align*}
u_{\RomanNumeralCaps{1}}(t,x,\xi) = \sum_{j=1}^{r}K_j(t,x)W_{j}(t,\xi).
\end{align*}
Plugging this representation as well as the quadratic flux of Burgers' equation into \eqref{eq:KstepGeneral} gives
\begin{align*}
\partial_t K_m(t,x) =& -\mathbb{E}\left[\partial_{x}f\left(u_{\RomanNumeralCaps{1}}(t,x,\cdot)\right)W_m(t,\cdot)\right] \\
=& -\frac12 \partial_x \mathbb{E}\left[ \sum_{i,j=1}^r K_i(t,x) W_i K_j(t,x) W_j W_m \right] \\
=& -\frac12\partial_x \Big(\sum_{i,j=1}^r K_i K_j \underbrace{\mathbb{E}\left[W_i W_j W_m\right]}_{=:a_{ijm}}\Big)
\end{align*}
Defining $\bm{K} = (K_1,\cdots,K_r)^T$ yields the differential equation
\begin{align}\label{eq:KStep1}
\partial_t K_m(t,x) = -\frac12\partial_x\left(\bm K^T \bm A_m \bm K\right) \qquad \text{ for } m = 1,\cdots,r
\end{align}
which according to Theorem~\ref{th:hyperbolicityKStep} is guaranteed to be hyperbolic. Representing the spatial coordinate with $N_x$ points, a number of $O(N_x\cdot r^3)$ evaluations per time step is needed for the evaluation of \eqref{eq:KStep1}. The terms $\bm A_m:=(a_{ijm})_{i,j=1}^r$ can be computed by applying a quadrature rule with $N_q$ weights $w_k$ and points $\xi_k$. Then, we can compute $W_i(t,\xi_k)=\sum_{\ell=0}^{N}\widehat{W}_{\ell i}\varphi_{\ell}(\xi_k)$ for $i=1,\cdots,r$ and $k = 1,\cdots,N_q$ in $O(r\cdot N \cdot N_q)$ operations. Furthermore, one can compute 
\begin{align}\label{eq:computationA}
    a_{ijm} = \sum_{k=1}^{N_q}w_k W_i(t,\xi_k) W_j(t,\xi_k) W_m(t,\xi_k)
\end{align}
in $O(N_q\cdot r^3)$ operations. Hence, the total costs for the $K$-step, which we denote by $C_K$ are
\begin{align*}
    C_K \lesssim N_t\cdot \max\left\{ N_x\cdot r^3,N_q\cdot r^3,r\cdot N \cdot N_q \right\}.
\end{align*}
It is important to point out the essentially quadratic costs of $r\cdot N \cdot N_q$. This non-linear term stems from the modal gPC approximation of the uncertain basis $W$. We will discuss a nodal approximation in Section~\ref{sec:nodal}, which yields linear costs. If we denote the memory requires for the $K$-step by $M_K$, we have
\begin{align*}
    M_K \lesssim \max\left\{ N_x\cdot r,N \cdot N_q \right\}
\end{align*}
when precomputing and storing the $N \cdot N_q$ terms $\varphi_i(\xi_k)$.
\subsection{$S$-step}
For the $S$-step, the basis functions $X_{\ell}$ and $W_i$ remain constant in time, i.e., we have
\begin{align}
u_{\RomanNumeralCaps{2}}(t,x,\xi)=\sum_{i,\ell=1}^r X_{\ell}(x) S_{\ell i}(t) W_i(\xi).
\end{align}
Plugging this representation into the $S$-step \eqref{eq:SstepGeneral} yields
\begin{align}
\dot{S}_{km}(t) =& \frac12\mathbb{E}\left[\left\langle \sum_{i,\ell=1}^r \sum_{j,q=1}^r \partial_x \left(X_{\ell}S_{\ell i} W_i X_{q} S_{q j} W_j\right)X_k\right\rangle W_m\right]\nonumber\\
=& \frac12\sum_{\ell=1}^r \sum_{q=1}^r \left\langle \partial_x\left(X_{\ell} X_{q}\right)X_k\right\rangle \cdot \mathbb{E}\left[\sum_{i =1}^r S_{\ell i} W_i \sum_{j=1}^r S_{q j}W_j W_m\right]\nonumber \\
=& \frac12\sum_{\ell=1}^r \sum_{q=1}^r \left\langle \partial_x\left(X_{\ell} X_{q}\right)X_k\right\rangle \cdot \mathbb{E}\left[L_{\ell} L_q W_m\right] \label{eq:SStep2}
\end{align}
Using a sufficiently accurate quadrature rule enables an efficient computation of the terms $\mathbb{E}\left[L_{\ell} L_q W_m\right]$. The values of $L_{q}(t,\xi_k)$ for $k=1,\cdots,N_q$ and $q=1,\cdots,r$ can be computed in $O(r\cdot N\cdot N_q)$ operations, since
\begin{align*}
L_{q}(t,\xi_k) = \sum_{m=1}^r W_m(t,\xi)S_{mq}(t)=\sum_{m=1}^r \sum_{i=0}^{N} \widehat W_{mi}(t) \varphi_i(\xi_k)S_{mq}(t)=\sum_{i=0}^{N} \widehat L_{qi} \varphi_i(\xi_k).
\end{align*}
The terms $\varphi_i(\xi_k)$ can be precomputed and stored. Furthermore, we make use of $\widehat L_{qi}(t,\xi_k):= \sum_{m=1}^r \widehat W_{mi}(t) \varphi_i(\xi_k)S_{mq}(t)$.
Then, a sufficiently accurate quadrature (e.g. a Gauss quadrature with $N_q = \lceil{1.5\cdot N+1}\rceil$) yields 
\begin{align*}
\mathbb{E}\left[L_{\ell} L_q W_m\right] = \sum_{k=1}^{N_q} w_k L_{\ell}(t,\xi_k)L_{q}(t,\xi_k)W_m(\xi_k)f_{\Xi}(\xi_k) \qquad \text{ for } \ell,q,m = 1,\cdots,r
\end{align*}
in $O(r^3\cdot N_q)$ operations. The derivative term in $\left\langle \partial_x\left(X_{\ell} X_{q}\right)X_k\right\rangle$ can be approximated with a finite volume stencil or a simple finite difference stencil. Dividing the spatial domain into $N_x$ points $x_1 < x_2 < \cdots<x_{N_x}$ lets us define elements $[x_j-\Delta x/2, x_j+\Delta x/2]$. Then, at spatial position $x_j$ we choose the finite difference approximation
\begin{align}\label{eq:derTerm1}
\partial_x\left(X_{\ell}(x) X_{q}(x)\right)\big|_{x=x_j} \approx \frac{1}{2\Delta x}\left( X_{\ell}(x_{j+1}) X_{q}(x_{j+1}) - X_{\ell}(x_{j-1}) X_{q}(x_{j-1})  \right).
\end{align}
Again choosing DG0 basis functions yields
\begin{align*}
\partial_x\left(X_{\ell}(x) X_{q}(x)\right)\big|_{x=x_j} \approx \frac{1}{2\Delta x^2}\left( \widehat{X}_{j+1,\ell} \widehat{X}_{j+1,q} - \widehat{X}_{j-1,\ell} \widehat{X}_{j-1,q}  \right).
\end{align*}
Summing over spatial cells to approximate the spatial integral $\left\langle \partial_x\left(X_{\ell} X_{q}\right)X_k\right\rangle$ gives
\begin{align}\label{eq:intTerm1}
\left\langle \partial_x\left(X_{\ell} X_{q}\right)X_k\right\rangle \approx \Delta x^{-3/2}\sum_{j=1}^{N_x}\frac{1}{2}\left( \widehat{X}_{j+1,\ell} \widehat{X}_{j+1,q} - \widehat{X}_{j-1,\ell} \widehat{X}_{j-1,q}  \right)\widehat X_{j,k}.
\end{align}
The required number of operations to compute this term for $\ell,q,k = 1,\cdots,r$ is $O(r^3\cdot N_x)$. Furthermore, the memory requirement is $O(r\cdot N_x)$ to store all $\widehat x_{j,\ell}$ as well as $O(r^3)$ to store all integral terms \eqref{eq:intTerm1}. Then the multiplication in \eqref{eq:SStep2} requires $O(r^4)$ operations. Hence the total costs and memory requirements for the $S$-step are 
\begin{align*}
    C_S \lesssim N_t\cdot r^3\cdot N_x \qquad \text{ and } \qquad M_S\lesssim \max\left\{r\cdot N_x, r^3\right\}.
\end{align*}


\subsection{$L$-step}
The final step is the $L$-step \eqref{eq:LstepGeneral}, which for Burgers' equation becomes
\begin{align}\label{eq:LStep1}
\partial_t L_k(t,\xi) &= \left\langle -\partial_x f(u_{\RomanNumeralCaps{3}}(t,\cdot,\xi))X_k(t,\cdot)\right\rangle\nonumber \\
&=\left\langle -\partial_x \left(\frac12\sum_{i,m=1}^r X_i(t,\cdot)L_i(t,\xi)X_m(t,\cdot)L_m(t,\xi)\right) X_k(t,\cdot)\right\rangle\nonumber\\
&=-\sum_{i,m=1}^r \frac12 L_i(t,\xi)L_m(t,\xi) \left\langle \partial_x  (X_i(t,\cdot)X_m(t,\cdot)) X_k(t,\cdot)\right\rangle.
\end{align}
To obtain a time evolution equation of the expansion coefficients $\widehat L_{ik}\in\mathbb{R}^{N+1\times r}$ such that 
\begin{align}
   L_k(t,\xi) \approx \sum_{i=0}^N \widehat L_{ik}(t)\varphi_i(\xi) 
\end{align}
we test \eqref{eq:LStep1} with $\varphi_i$, which gives
\begin{align}\label{eq:Lstep2}
    \partial_t \widehat{L}_{\ell k}(t) =-\sum_{i,m=1}^r \frac12\mathbb{E}\left[L_i(t,\cdot)L_m(t,\cdot)\varphi_{\ell}\right] \left\langle \partial_x  (X_i(t,\cdot)X_m(t,\cdot)) X_k(t,\cdot)\right\rangle.
\end{align}
The terms $\mathbb{E}\left[L_i(t,\cdot)L_m(t,\cdot)\varphi_{\ell}\right]$ can be computed analogously to \eqref{eq:computationA} in $O(r^2\cdot N\cdot N_q)$ operations, taking up memory of $O(N\cdot r^2)$. Precomputing the terms $\varphi_i(\xi_k)$ again has memory requirements of $O(N\cdot N_q)$. The term $\left\langle \partial_x  (X_i(t,\cdot)X_m(t,\cdot)) X_k(t,\cdot)\right\rangle$ can be reused from the $S$-step computation and the multiplication in \eqref{eq:Lstep2} requires $O(N\cdot r^3)$ operations. Therefore, the overall costs and memory requirements for the $L$-step are
\begin{align*}
    C_L \lesssim N_t\cdot r^2\cdot N\cdot N_q \qquad \text{ and } \qquad M_L \lesssim \max\left\{N\cdot r^2, N\cdot N_q\right\}.
\end{align*}
Note that the QR decompositions needed to compute $S,X,W$ from $K$ and $L$ require $O(r^2\cdot N_x)$ as well as $O(r^2\cdot N)$ operations respectively in every time step.

\subsection{Filtered Matrix projector-splitting integrator}\label{sec:FilterDLR}
Similar to filtered stochastic-Galerkin, we wish to apply a filtering step in between time steps. Let us write the low-rank solution as
\begin{align*}
    &u_r(t,x,\xi) = \sum_{\ell=1}^r X_{\ell}(t,x) \sum_{i=0}^N \widehat{L}_{i\ell}(t)\varphi_i(\xi) = \sum_{i=0}^N \alpha_{i}(t,x)\varphi_i(\xi),\\
    &\text{where } \bm{\alpha}(t,x) := \left(\sum_{\ell=1}^r X_{\ell}(t,x)\widehat{L}_{i\ell}\right)_{i=0}^N.
\end{align*}
Following the idea of filtering, we now wish to determine $\bm\alpha$ such that the solution representation minimizes the L$^2$-error in combination with a term punishing oscillatory solution values. Equivalently to the derivation of the fSG ansatz \eqref{eq:filterOptProblem}, this gives
\begin{align*}
    \bm{\widehat{\alpha}} = \argmin_{\bm{\alpha}\in\mathbb{R}^{N+1}}\left\{\mathbb{E}\left[ \left(\bm{\alpha}^T\bm{\varphi} - u_{ex}\right)^2 \right] + \lambda \mathbb{E}\left[\left( \mathcal{L}\bm{\alpha}^T\bm{\varphi}\right)^2\right]\right\}.
\end{align*}
As before we have 
\begin{align*}
    \widehat{\alpha}_i = \frac{1}{1+\lambda i^2 (i-1)^2}\mathbb{E}[u_{ex}\varphi_i].
\end{align*}
Hence, the filtered expansion coefficients are again the original expansion coefficients multiplied by a dampening factor $g_{\lambda}(i):=\frac{1}{1+\lambda i^2 (i-1)^2}$. In order to preserve the low-rank structure, we apply this factor to the coefficients $\widehat{L}_{i\ell}$, i.e., the filtered coefficients are given by $\widetilde{L}_{i\ell} = g_{\lambda}(i)\widehat{L}_{i\ell}$. As for fSG, the filter is applied after every full time step. I.e., for the matrix projector-splitting integrator, the filter is applied after the $L$-step and for the unconventional integrator, we filter after the $S$-step. Here, one can determine a filtered $W^1$ after the $S$-step by a QR factorization of the filtered coefficients $\widetilde{L}_{i\ell}$ or apply the filter directly on $W^1$.

\subsection{Nodal discretization}\label{sec:nodal}
Note that the $L$-step has cost $C_L \lesssim N_t\cdot r^2\cdot N\cdot N_q$. To compute all arising integrals, the number of Gauss quadrature points must be chosen as $N_q = \left\lceil \frac32 N -1 \right\rceil$. Hence, the number of gPC polynomials goes into the costs quadratically. To guarantee linearity with respect to $N$, a nodal (or collocation) discretization can be chosen for the random domain. Hence, the functions $L_k$ are described on a fixed set of collocation points $\xi_1,\cdots,\xi_{N_q}\in\Theta$, leading to the discrete function values $L_{qk}(t) := L_k(t,\xi_q)$. Then, the $L$-step \eqref{eq:LStep1} can be written as
\begin{align}\label{eq:LStepNodal}
\partial_t L_{qk}(t) =-\sum_{i,m=1}^r \frac12 L_{qi}L_{qm} \left\langle \partial_x  (X_i(t,\cdot)X_m(t,\cdot)) X_k(t,\cdot)\right\rangle.
\end{align}
In this case, the numerical costs become $C_L^{\text{nodal}} \lesssim N_t\cdot r^3\cdot N_q$, i.e., the number of quadrature points affects the costs linearly. Picking the collocation points according to a quadrature rule, integral computations over the random domain can be computed efficiently. When replacing the modal $L$-step \eqref{eq:Lstep2} by its nodal approximation \eqref{eq:LStepNodal}, the $K$, $S$ and $L$ equations of DLRA essentially gives the dynamical low-rank analogue to the stochastic-Collocation method. Note however that in contrast to stochastic-Collocation, the derived nodal method for DLRA will be intrusive, since new equations need to be derived and the different quadrature points couple in every time step through integral evaluations. Furthermore, the application of filters becomes more challenging, since the gPC expansion coefficients of $L_k$ are unknown. Computing these coefficients is possible but again leads to quadratic costs with respect to the number of basis functions.

\subsection{Extension to multiple dimensions}\label{sec:multiD}
Note that one of the key challenges facing uncertainty quantification is the curse of dimensionality and the resulting uncontrollable growth in the amount of data to be stored and treated. To outline how the dynamical low-rank method tackles this challenge, we now focus on discussing the applicability, costs and memory requirements of DLRA with higher dimensional uncertainties. 
A naive extension to multi-D can be derived by including additional uncertainties in the $W_i$ basis. For two-dimensional uncertainties with probability density $f_{\Xi}(\xi_1,\xi_2)=f_{\Xi_1}(\xi_1)f_{\Xi_2}(\xi_2)$, where $\xi_i\in\Theta_i$, the modal approach for the representation of the basis \eqref{eq:basisW} becomes 
\begin{align}\label{eq:basisWMultiD}
    W_i(t,\xi_1,\xi_2) = \sum_{m=0}^N\sum_{k=0}^N \widehat{W}_{mki}(t)\varphi_m(\xi_1)\varphi_k(\xi_2) \qquad \text{ for } i = 1,\cdots,r.
\end{align}
In this case, the $L$-step reads
\begin{align*}
    \partial_t L_k(t,\xi_1,\xi_2) = -\sum_{i,m=1}^r \frac12 L_i(t,\xi_1,\xi_2)L_m(t,\xi_1,\xi_2) \left\langle \partial_x  (X_i(t,\cdot)X_m(t,\cdot)) X_k(t,\cdot)\right\rangle.
\end{align*}
An evolution equation for the modal expansion coefficients is obtained by testing against $\varphi_{\alpha}(\xi_1)$ and $\varphi_{\beta}(\xi_2)$, yielding
\begin{align*}
    \partial_t \widehat{L}_{\alpha \beta k}(t) = -\sum_{i,m=1}^r \frac12 \mathbb{E}[L_i L_m \varphi_{\alpha}\varphi_{\beta}] \left\langle \partial_x  (X_i(t,\cdot)X_m(t,\cdot)) X_k(t,\cdot)\right\rangle.
\end{align*}
An extension to higher dimensions is straight forward. For our naive treatment of uncertainties with dimension $p$, we have
\begin{align*}
    C_K \lesssim N_t\cdot \max\left\{ N_x\cdot r^3,N_q^p\cdot r^3,r\cdot N^p \cdot N_q^p \right\}\qquad &\text{ and } \qquad M_K \lesssim \max\left\{ N_x\cdot r,N \cdot N_q \right\},\\
    C_S \lesssim N_t\cdot r^3\cdot N_x \qquad &\text{ and } \qquad M_S\lesssim \max\left\{r\cdot N_x, r^3\right\},\\
    C_L \lesssim N_t\cdot r^2\cdot N^p\cdot N_q^p \qquad &\text{ and } \qquad M_L \lesssim \max\left\{N^p\cdot r^2, N\cdot N_q\right\}.
\end{align*}
Note that computational requirements become prohibitively expensive for large $p$. The increased numerical costs can be reduced by further splitting the uncertain domain \cite{lubich2013dynamical,lubich2015time,lubich2015timetensor,ceruti2020time}, which we will leave to future work.

\section{Boundary conditions}\label{sec:BC}
So far, we have not discussed how to treat boundary conditions. In this work, we focus on Dirichlet boundary conditions
\begin{align*}
    u_r(t,x_L,\xi) = u_L(\xi) \qquad \text{and} \qquad u_r(t,x_R,\xi) = u_R(\xi).
\end{align*}
Note that Dirichlet values for Burgers' equation are commonly constant in time, which is why we omit time dependency here. A straightforward way to impose boundary conditions is to project the boundary condition onto the DLRA basis functions \cite{sapsis2009dynamically,sapsis2012dynamical}. However, if the low-rank basis cannot represent the boundary condition, the derived projection will not exactly match the imposed Dirichlet values, leading to an error. In this case, we have
\begin{align*}
    u_L(\xi) &\neq \sum_{i=1}^{r} \mathbb{E}\left[ u_r(t,x_L,\cdot) W_{i}(t,\cdot) \right] W_{i}(t,\xi), \\
    u_R(\xi) &\neq \sum_{i=1}^{r} \mathbb{E}\left[ u_r(t,x_R,\cdot) W_{i}(t,\cdot) \right] W_{i}(t,\xi).
\end{align*}
 Therefore, we now discuss how to preserve certain basis functions, which exactly represent the solution at the boundary. The strategy is similar to \cite{einkemmer2021mass}, where basis functions are preserved to guarantee conservation properties. The idea of omitting a constant basis function is also used in the Dynamically Orthogonal (DO) method \cite{sapsis2009dynamically}. In \cite{musharbash2018dual} a method to enforce boundary conditions has been proposed for DO systems. We propose a similar approach to impose boundary conditions for the DLRA approximation of Burgers' equation. Following \cite{musharbash2018dual} we start by modifying the original ansatz \eqref{eq:rankrsol} to
\begin{align}\label{eq:ansatzConserved}
    u(t,x,\xi) \approx u_c(t,x,\xi):= \sum_{i=1}^{N_c} \hat u_i(t,x)V_i(\xi) + \sum_{i,\ell=1}^r X_{\ell}(t,x) S_{\ell i}(t) W_i(t,\xi),
\end{align}
where $V_i$ and $W_i$ form an orthonormal set of basis functions. We aim to preserve the $N_c$ basis functions $V_i:\Theta\rightarrow\mathbb{R}$, which are chosen such that
\begin{align*}
    u_L(\xi) = \sum_{i=1}^{N_c} \mathbb{E}\left[ u_L V_i \right] V_i(\xi), \quad\text{ and }\quad u_R(\xi) = \sum_{i=1}^{N_c} \mathbb{E}\left[ u_R V_i \right] V_i(\xi).
\end{align*}
To ensure $\langle W_i(t,\cdot) V_k \rangle = 0$ for all times $t\in\mathbb{R}_+$, we choose a modal representation
\begin{align}\label{eq:lowrankWspan}
    W_i(t,\xi) = \sum_{j=1}^{N+1}\widehat{W}_{ji}(t)P_j(\xi),
\end{align}
where the orthonormal basis functions $P_j(\xi)$ are constructed such that $\mathbb{E}\left[ P_j V_k \right] = 0$. This can be done by generating the basis functions $P_j$ with Gram-Schmidt and including the Dirichlet values of $u_L$ and $u_R$ as first two functions into the process of generating the basis. I.e., we take $V_1 = \widetilde V_1/\mathbb{E}[\widetilde V_1^2]$ and $V_2 = \widetilde V_2/\mathbb{E}[\widetilde V_2^2]$ with
\begin{align*}
    \widetilde V_1(\xi) = u_L(\xi), \quad \text{and }\enskip \widetilde V_2(\xi) = u_R(\xi) - \frac{\mathbb{E}[\widetilde V_1 u_R]}{\mathbb{E}[\widetilde V_1^2]}.
\end{align*}
The functions to generate the low-rank basis $W_i$ according to \eqref{eq:lowrankWspan} are then computed with
\begin{align*}
    \widetilde P_i(\xi) = \varphi_{i-1}(\xi) - \mathbb{E}[V_1 \varphi_{i-1}]- \mathbb{E}[V_2 \varphi_{i-1}] - \sum_{j=1}^{i-1} \frac{\mathbb{E}[\widetilde P_j \varphi_{i-1}]}{\mathbb{E}[\widetilde P_j^2]}
\end{align*}
as $P_i = \widetilde P_i/\mathbb{E}[\widetilde P_i^2]$. Note that in this case $N_c=2$ is sufficient. Then, evolution equations for $\hat u_i(t,x)$ can be derived by taking moments of the original system \eqref{eq:hyperbolicProblem}, i.e.,
\begin{align}
\partial_t \hat u_i(t,x) + \partial_x \mathbb{E}\left[f(u_c(t,x,\xi))V_i \right] = 0\qquad\text{for } i = 1,\cdots,N_c.
\end{align}
The low-rank part of \eqref{eq:ansatzConserved} is then solved with a classical dynamical low-rank method. Now, to ensure that the ansatz \eqref{eq:ansatzConserved} matches the imposed boundary values, we must prescribe the condition
\begin{align*}
    K_i(t,x_L) &= \mathbb{E}\left[ u_c(t,x_L,\xi) W_i(t,\cdot) \right] = \mathbb{E}\left[ u_L W_i(t,\cdot) \right] = 0, \\
    \hat u_i(t,x_L) &= \mathbb{E}\left[ u_c(t,x_L,\xi) V_i \right] = \mathbb{E}\left[ u_L V_i \right].
\end{align*}
Similarly, for the right boundary we have
\begin{align*}
    K_i(t,x_R)  = 0, \quad\text{ and } \quad \hat u_i(t,x_R) = \mathbb{E}\left[ u_R V_i(t,\cdot) \right].
\end{align*}

\section{Numerical discretization}\label{sec:discretize}
As discussed in Theorem~\ref{th:hyperbolicityKStep}, the $K$-step equation is hyperbolic, meaning that it can be discretized with a finite volume or DG method. Since we have taken DG0 elements to discretize the spatial domain, we will now derive a DG0 method. The derivation will be demonstrated for the modal matrix projector-splitting integrator. The extension to nodal methods and the unconventional integrator are straight forward. To simplify notation, let us collect the spatial expansion coefficients $\widehat K_{jm}$ in the vector $\bm{\widehat K}_{j} = \left(\widehat{K}_{j1},\cdots,\widehat{K}_{jr}\right)^T$ as well as the $K$ variables in a vector $\bm{K}_{j} = \left(K_{j1},\cdots,K_{jr}\right)^T$. Then, taking the $K$-step equation \eqref{eq:KStep1} and testing with $Z_j$ gives
\begin{align*}
&\partial_t \left\langle\sum_{i=1}^{N_x}\widehat{K}_{im}Z_i Z_j\right\rangle = -\frac12\left\langle\partial_x\left(\bm{K}^T \bm A_m \bm{K}\right)  Z_j\right\rangle  \\
\Leftrightarrow \; &\partial_t \widehat{K}_{jm} = -\frac12\left[\bm{K}^T \bm A_m \bm{K} Z_j\right]_{x_{j-1/2}}^{x_{j+1/2}}-\frac12\int_{x_{j-1/2}}^{x_{j+1/2}} \bm{K}_{i}^T \bm A_m \bm{K}_{\ell}\underbrace{\partial_x Z_j}_{=0}\,dx.
\end{align*}
Here, we used that the support of $Z_j$ is restricted to the spatial cell $[x_{j-1/2},x_{j+1/2}]$. Note that the right hand side requires knowing the vector $\bm K$ at the cell interfaces. These values are approximated with a numerical flux $\bm g:\mathbb{R}^{r}\times\mathbb{R}^r\rightarrow\mathbb{R}^r$, i.e., we have
\begin{align*}
    \frac12\left.\bm{K}^T \bm A_m \bm{K}\right\vert_{x_{j+1/2}} \approx g_m(\bm K_j^n,\bm K_{j+1}^n).
\end{align*}
Choosing the Lax-Friedrichs flux
\begin{align*}
g_m(\bm K_j^n,\bm K_{j+1}^n) = \frac14 \left( \bm K_j^n+\bm K_{j+1}^n\right)^T\bm A_m \left( \bm K_j^n+\bm K_{j+1}^n\right) - \frac{\Delta x}{2\Delta t}\left(K_{j+1,m}^n-K_{j,m}^n\right)
\end{align*}
and remembering that $Z_j(x) = \Delta x^{-1/2}\chi_{[x_{j-1/2},x_{j+1/2}]}(x)$ gives
\begin{align*}
\partial_t\bm{\widehat K}_j = -\frac{1}{\sqrt{\Delta x}}\left( \bm g(\bm K_j,\bm K_{j+1})-\bm g(\bm K_{j-1},\bm K_{j}) \right).
\end{align*}
When choosing a forward Euler time discretization and using the notation $\bm{K}_j^n:=\bm{K}_j(t_n)$ as well as $\bm{\widehat K}_j^n:=\bm{\widehat K}_j(t_n)$ one obtains
\begin{align*}
\bm{\widehat K}_j^{n+1} = \bm{\widehat K}_j^n-\frac{\Delta t}{\sqrt{\Delta x}}\left( \bm g(\bm K_j^n,\bm K_{j+1}^n)-\bm g(\bm K_{j-1}^n,\bm K_{j}^n) \right).
\end{align*}
Slope limiters as well as higher order time discretizations can be chosen to obtain a more accurate discretization. The remaining $S$ and $L$ steps are discretized with a forward Euler discretization. Hence, when defining
\begin{align*}
    &\widetilde{X}_{\ell,q,k}^n = \Delta x^{-3/2}\sum_{j=1}^{N_x}\frac{1}{2}\left( \widehat{X}_{j+1,\ell}^n \widehat{X}_{j+1,q}^n - \widehat{X}_{j-1,\ell}^n \widehat{X}_{j-1,q}^n  \right)\widehat X_{j,k}^n, \\
    &\widetilde{W}_{\ell,q,m}^n = \sum_{k=1}^{N_q} w_k L_{\ell}^n(\xi_k)L_{q}^n(\xi_k)W_m^n(\xi_k)f_{\Xi}(\xi_k),\\
    &\widetilde{L}_{\ell,q,m}^n = \sum_{k=1}^{N_q} w_k L_{\ell}^n(\xi_k)L_{q}^n(\xi_k)\varphi_m(\xi_k)f_{\Xi}(\xi_k),
\end{align*}
the time update formulas of $S$ and $L$ step can be written as
\begin{align*}
    S^{n+1}_{km} =&  S^{n}_{km} + \Delta t \sum_{\ell,q = 1}^r\widetilde{X}_{\ell,q,k}^n \cdot \widetilde{W}_{\ell,q,m}^n, \\
    \widehat L^{n+1}_{mk} =& \widehat{L}^{n}_{mk} - \Delta t \sum_{\ell,q = 1}^r\widetilde{X}_{\ell,q,k}^n \cdot \widetilde{L}_{\ell,q,m}^n.
\end{align*}
Note that here, we are using stabilizing terms in the $K$-step which are not applied to the $S$- and $L$-steps. To obtain a consistent discretization of all three steps, we can also include such stabilizing terms which take the form
\begin{align}
\bar{S}_{km}^n = \frac12\sum_{\ell = 1}^r\sum_{j=1}^{N_x} \widehat X_{jk}^n \left(\widehat X_{j+1,\ell}^n+\widehat X_{j-1,\ell}^n\right) S_{\ell m}^n - S_{km}^n, \\
\bar{L}_{mk}^n = \frac12\sum_{\ell = 1}^r\sum_{j=1}^{N_x} \widehat X_{jk}^n \left(\widehat X_{j+1,\ell}^n+\widehat X_{j-1,\ell}^n\right) L_{\ell m}^n - L_{mk}^n.
\end{align}
Including these terms in the $S$- and $L$-steps yields the modified time updates
\begin{subequations}\label{eq:stableUpdate}
\begin{align}
    S^{n+1}_{km} =&  S^{n}_{km} + \Delta t \sum_{\ell,q = 1}^r\widetilde{X}_{\ell,q,k}^n \cdot \widetilde{W}_{\ell,q,m}^n - \bar{S}_{km}^n, \\
    \widehat L^{n+1}_{mk} =& \widehat{L}^{n}_{mk} - \Delta t \sum_{\ell,q = 1}^r\widetilde{X}_{\ell,q,k}^n \cdot \widetilde{L}_{\ell,q,m}^n+\bar{L}_{mk}^n.
\end{align}
\end{subequations}
Note that this method is consistent in that it includes stabilizing terms which appear when applying the matrix projector-splitting integrator to the fully discretized system. I.e., writing down a stable discretizing of the original equations in $x$ and $\xi$, which gives a matrix differential equation $\bm{\dot{y}}(t) = \bm F(\bm y)$ with $\bm y,\bm F \in\mathbb{R}^{N_x\times N_q}$ and applying the matrix projector-splitting integrator to this system will give the same equations as the stabilized update \eqref{eq:stableUpdate}. In our numerical experiments we always use the stabilized update. We observed that the stabilization terms do not need to be applied for the matrix projector-splitting integrator. However, the unconventional integrator yields poor results if the stabilizing terms are left out.

\section{Numerical results}\label{sec:numResults}
In this section, we represent numerical results to the equations and strategies derived in this work. We make the code to reproduce all results presented in the following available in \cite{code}. Our implementation is used to investigate Burgers' equation \eqref{eq:Burgers} with an uncertain initial condition
\begin{align}\label{eq:IC2}
u_{\text{IC}}(x,\xi) &:= 
\begin{cases} u_L, & \mbox{if } x< x_0+\sigma_1\xi_1 \\
u_R+\sigma_2\xi_2, & \text{else }
\end{cases}.
\end{align}
The initial condition is a shock with an uncertain shock position $x_0+\sigma_1\xi_1$ where $\xi\sim U(-1,1)$ and an uncertain right state $u_R+\sigma_2\xi_2$, where $\xi_2\sim U(0,1)$. At the boundary, we impose Dirichlet values $u_L$ and $u_R(\xi) = u_R+\sigma_2\xi_2$. We choose a CFL condition $\Delta t = CFL\cdot \Delta x / u_L$, where $CFL = 0.5$. The remaining parameter values are:\\
\begin{center}
    \begin{tabular}{ | l | p{7cm} |}
    \hline
    $[x_L,x_R]=[0,1]$ & range of spatial domain \\
    $N_x=600$ & number of spatial cells \\
    $t_{end}=0.01$ & end time \\
    $x_0 = 0.5, u_L = 12, u_R = 1, \sigma_1 = 0.2, \sigma_2 = 5$ & parameters of initial condition \eqref{eq:IC2}\\
    $r = 9, (N+1)^2 = 100, N_q = 256$ & rank, number of moments and quadrature points for DLRA \\
    \hline
    \end{tabular}
\end{center}
The uncertain basis functions are chosen to be the tensorized gPC polynomials with maximum degree of up to order $9$, i.e., we have $10^2$ basis functions. In this setting, we investigate the full SG solution which uses the same basis, since using 10 moments in every dimension is a reasonable choice to obtain satisfactory results. Furthermore, when picking a total number of $r=9$ moments, i.e., choosing total degree $2$ gPC polynomials in every dimension leads to a poor approximation, which is dominated by numerical artifacts. As can be seen in Figure~\ref{fig:ExpVarDLRmodal2D}, the modal DLRA method agrees well with the finely resolved SG solution when using the matrix projector-splitting integrator, especially for the variance. Compared to the matrix projector-splitting integrator, the unconventional integrator shows an improved approximation of the expected value while leading to dampening of the variance. For a clearer picture of this effect, see Figure~\ref{fig:L2Errors}. All in all we observe a heavy reduction of basis functions to achieve satisfactory results by the use of dynamical-low rank approximation. Results computed by the DLRA method match nicely with the finely resolves stochastic-Galerkin method. Thus, we can conclude that DLRA provides an opportunity to battle the curse of dimensionality, since the required memory to achieve a satisfactory solution approximation grows moderately with dimension. When employing tenor approximations as presented in \cite{ceruti2020time}, we expect linear instead of exponential growth with respect to the dimension. However, we leave an extension to higher uncertain domains in which this strategy becomes crucial to future work.
\begin{figure}[H]
    \centering
    \includegraphics[width=0.99\textwidth]{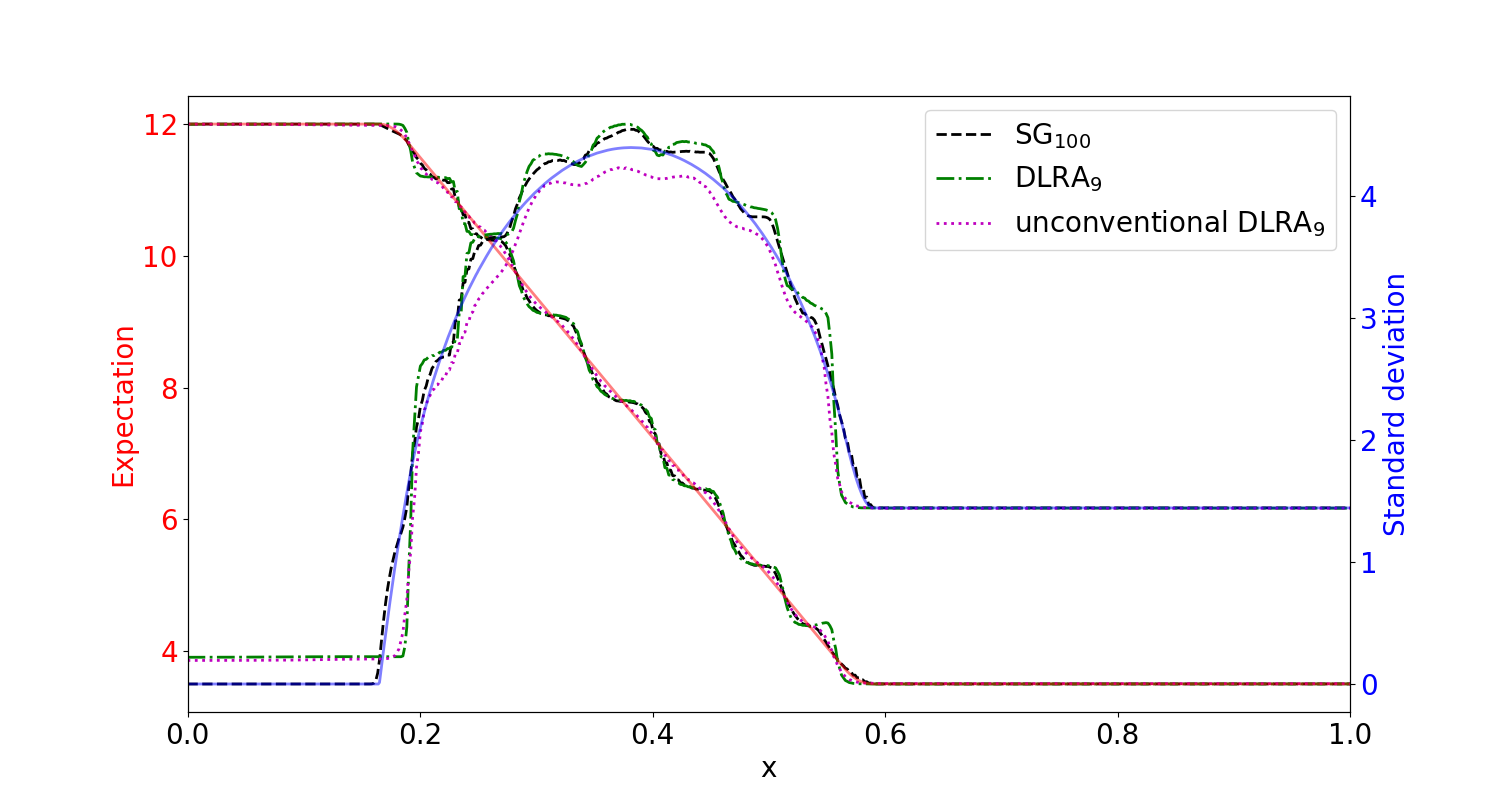}
    \caption{Expectation and variance computed with DLRA (matrix projector-splitting as well as unconventional integrator) and SG method. The rank is $r=9$ and both DLRA and SG use $100$ gPC basis functions as modal discretization. Integrals are computed using $256$ quadrature points, which allows an exact computation of all integral terms.}
    \label{fig:ExpVarDLRmodal2D}
\end{figure}
\begin{figure}[H]
    \centering
    \includegraphics[width=0.99\textwidth]{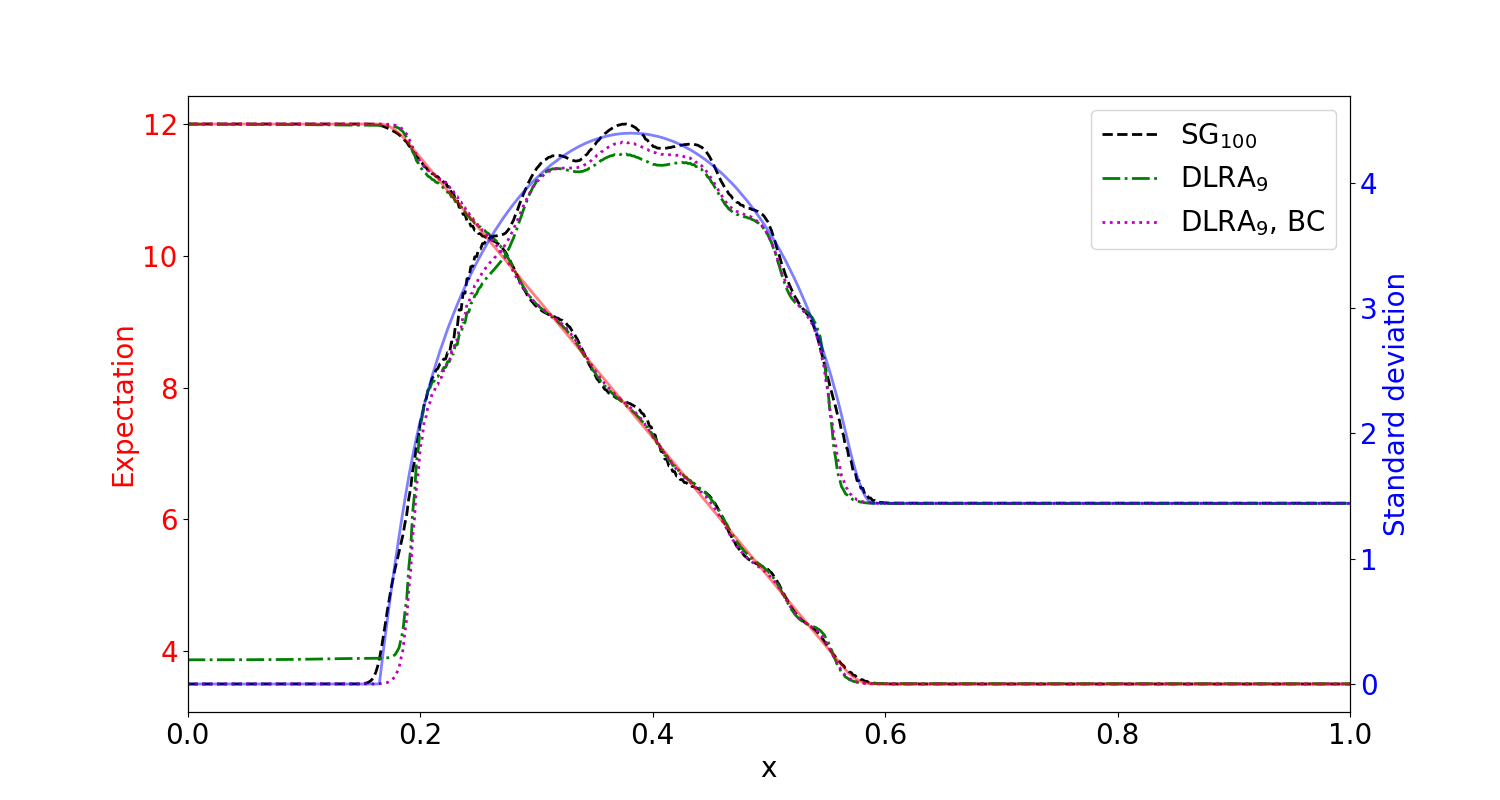}
    \caption{Expectation and variance computed with SG as well as DLRA (unconventional integrator) with and without boundary condition fix according to Section~\ref{sec:BC}. The rank is $r=9$ and both DLRA and SG use $100$ gPC basis functions as modal discretization. Integrals are computed using $256$ quadrature points, which allows an exact computation of all integral terms.}
    \label{fig:ExpVarDLRmodal2DBC}
\end{figure}
In Figure~\ref{fig:ExpVarDLRmodal2DBC}, we present numerical results for the proposed strategy to impose Dirichlet boundary conditions. As a comparison, we include the previous stochastic-Galerkin result as well as the DLRA result when using the unconventional integrator. Taking a look at the left boundary, we observe a violation of the imposed Dirichlet values by the standard DLRA method. We observe that the dynamical low-rank basis cannot represent deterministic solutions, since the constant basis function in $\xi$ will be lost during the computation. According to Section~\ref{sec:BC}, we fix the constant basis as well as the linear basis in $\xi_2$. The remainder is represented with a low-rank ansatz of rank $r=9$ using $10^2-2 = 98$ basis functions in $\xi$. Again, we use the unconventional integrator, which gives the results depicted in Figure~\ref{fig:ExpVarDLRmodal2DBC}. It can be seen that the proposed strategy allows for an exact representation of the chosen Dirichlet values. Furthermore, the strategy improves the solution representation and shows improved agreement with the exact solution.

Let us now turn to the filtered SG method and apply it to the two-dimensional test case. Here, a parameter study leads to an adequate filter strength of $\lambda=10^{-5}$. Taking a look at the resulting fSG approximation in Figure~\ref{fig:ExpVarDLRmodal2DFilters}, we observe a significant improvement of the expected value approximation through filtering. The variance, though dampened by the filter, shows less oscillations and qualitatively agrees well with the exact variance. Note that the use of high-order filters can mitigate dampening effects of the variance, see e.g. \cite{kusch2020oscillation}, however we leave the study of different filters to future work. When comparing the fSG solution with the low-rank methods making use of the same filter as fSG, one sees a close agreement with the finely resolved fSG solution. Note that the unconventional integrator again shows a dampened variance approximation.
\begin{figure}[H]
    \centering
    \includegraphics[width=0.99\textwidth]{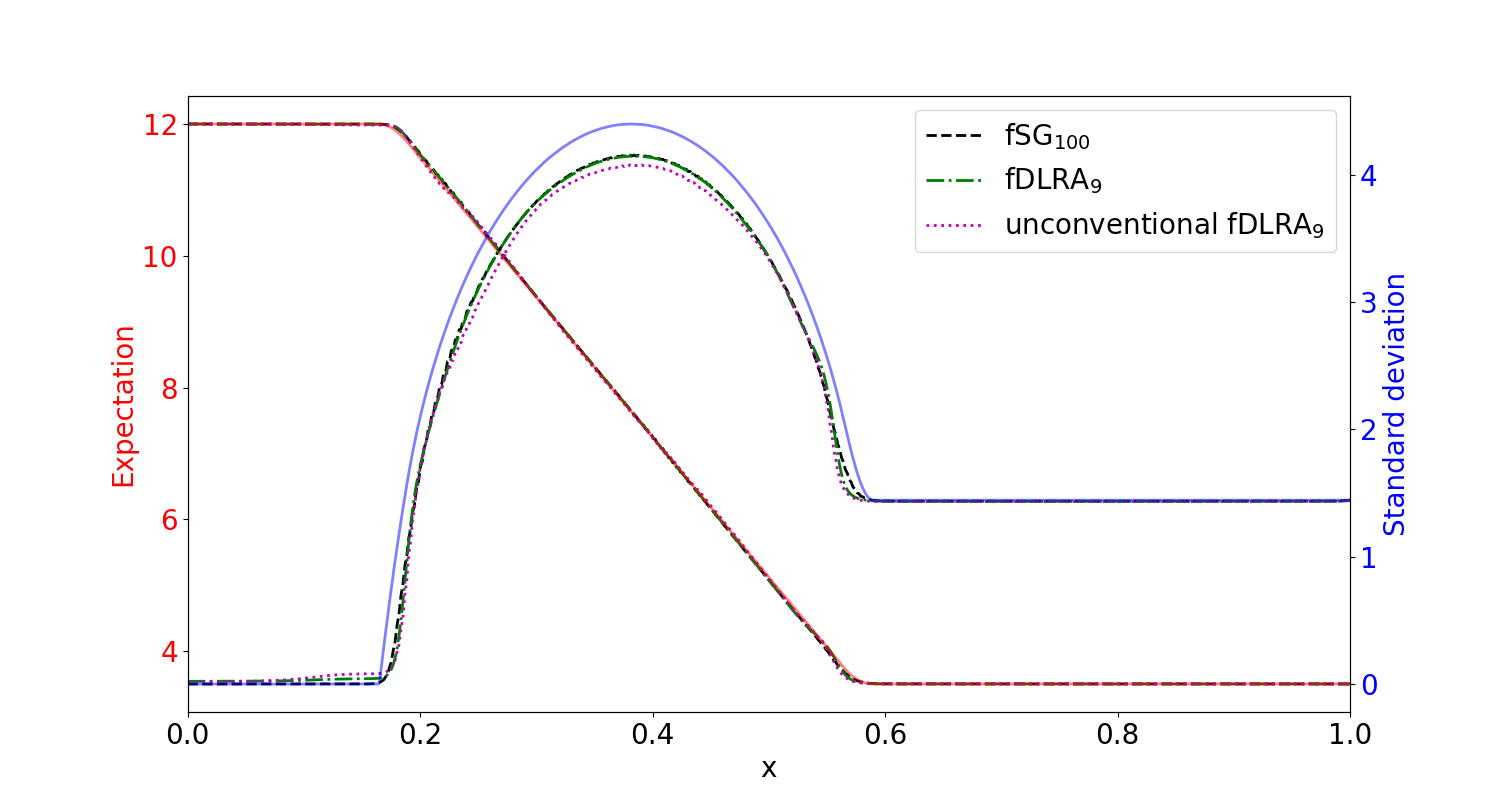}
    \caption{Expectation and variance computed with the DLRA and SG method using the L$^2$ filter. The rank is $25$ and both DLRA and SG use $400$ gPC basis functions. A filter strength of $\lambda=10^{-5}$ is chosen.}
    \label{fig:ExpVarDLRmodal2DFilters}
\end{figure}
Figure~\ref{fig:solutions2DFixedXi} gives a better impression of how the different modal methods behave in the random domain by showing the solution at a fixed spatial position $x^* = 0.42$. The exact solution, which is depicted in Figure~\ref{fig:referenceSolutionFixedXi}, shows a discontinuity in the $\xi_1$-domain as well as a partially linear profile in the $\xi_2$-domain. The modal DLRA method using rank $9$, depicted in Figure~\ref{fig:2DXiSG} agrees well with the SG$_{100}$ method, which is shown in Figure~\ref{fig:2DXiDLR}. Comparing the filtered SG$_{100}$ solution in Figure~\ref{fig:2DXifSG} and the filtered DLRA (fDLRA) solution with rank $9$, shown in Figure~\ref{fig:2DXifDLR}, one sees that again both methods lead to almost identical solutions. As expected, while dampening oscillations, the filter smears out the shock approximation.
\newpage
\newgeometry{top=20mm}
\begin{figure}[H]
	\centering
	\begin{subfigure}{0.5\linewidth}
		\centering
		\includegraphics[width=1.0\linewidth]{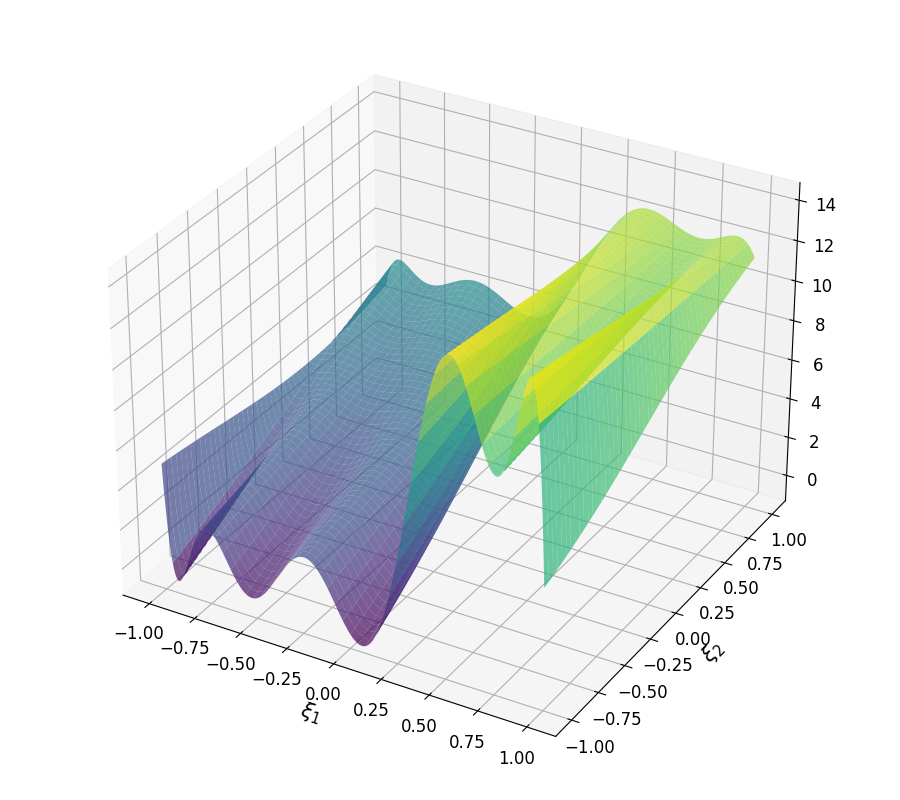}
		\caption{SG$_{100}$}
		\label{fig:2DXiSG}
	\end{subfigure}%
	\hfill
	\begin{subfigure}{0.5\linewidth}
		\centering
		\includegraphics[width=1.0\linewidth]{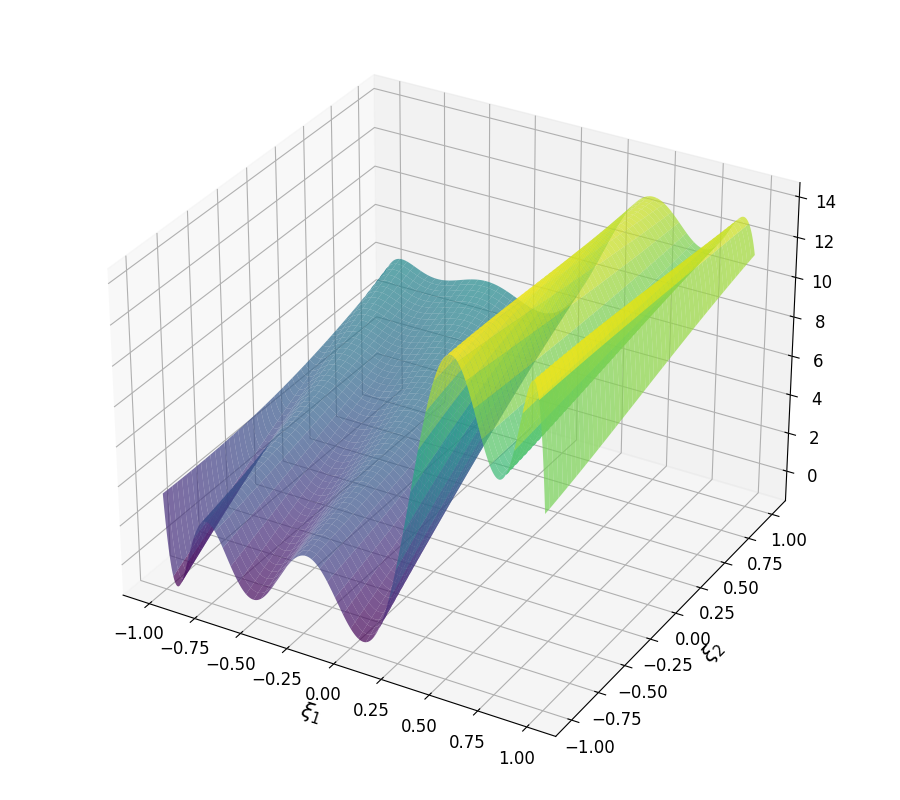}
		\caption{DLRA$_{9}$}
		\label{fig:2DXiDLR}
	\end{subfigure}\\
	\begin{subfigure}{0.5\linewidth}
		\centering
		\includegraphics[width=1.0\linewidth]{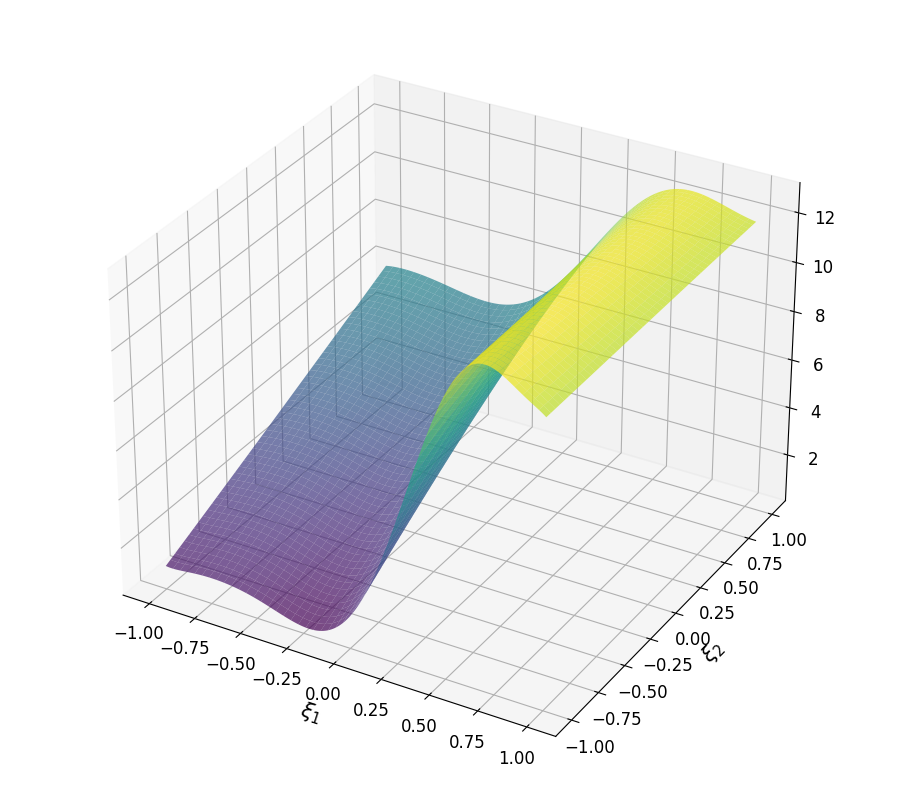}
		\caption{fSG$_{100}$}
		\label{fig:2DXifSG}
	\end{subfigure}%
	\hfill
	\begin{subfigure}{0.5\linewidth}
		\centering
		\includegraphics[width=1.0\linewidth]{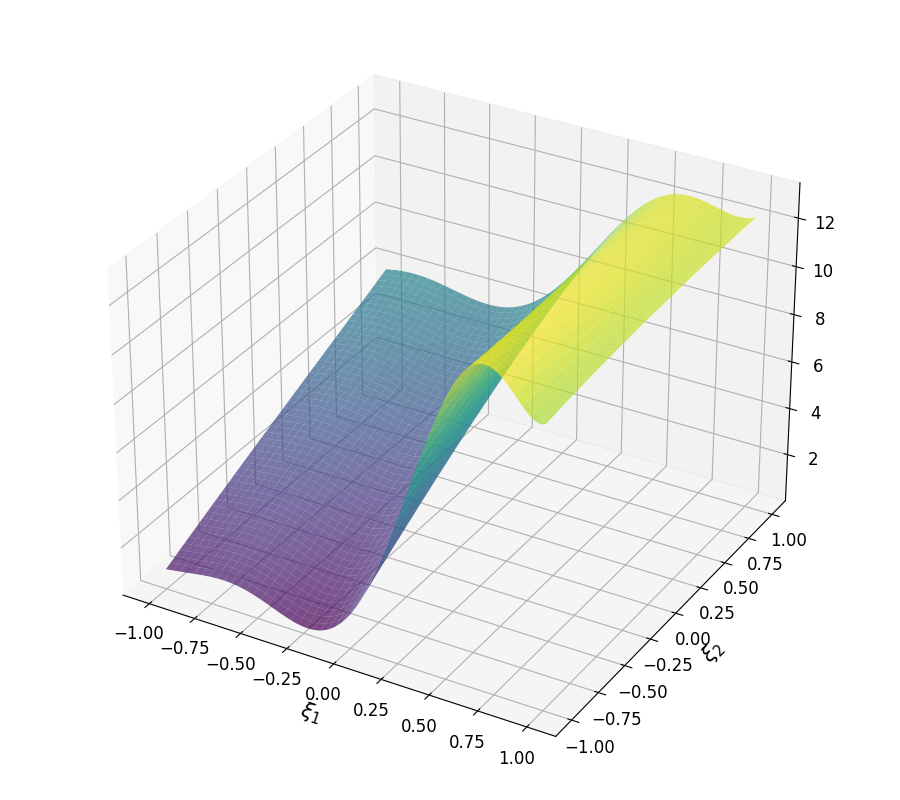}
		\caption{fDLRA$_{9}$}
		\label{fig:2DXifDLR}
	\end{subfigure}\\
	\begin{subfigure}{0.5\linewidth}
		\centering
		\includegraphics[width=1.0\linewidth]{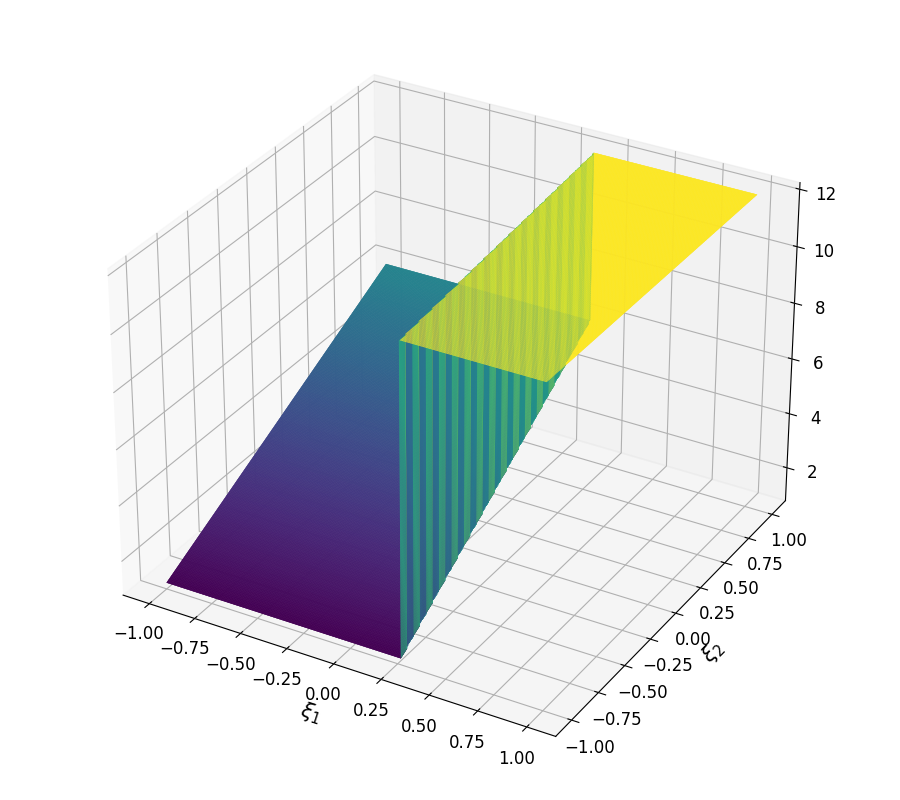}
		\caption{exact}
		\label{fig:referenceSolutionFixedXi}
	\end{subfigure}
	\hfill
	\caption{Results for SG and DLRA using the matrix projector-splitting with and without filters at fixed spatial position $x^* = 0.42$.}
	\label{fig:solutions2DFixedXi}
\end{figure}
\restoregeometry
The nine uncertain basis functions $W_i(t,\bm\xi)$ generated by the modal DLRA$_9$ method when using the matrix projector-splitting integrator at the final time $t_{end}=0.01$ are depicted in Figure~\ref{fig:BasisWFig}. Opposed to SG$_{100}$ which uses gPC basis functions of maximum degree up to $10$ to represent the uncertain domain, the DLRA method picks a set of nine basis functions, which efficiently represent the uncertain domain.
\begin{figure}[H]
    \centering
    \includegraphics[width=0.99\textwidth]{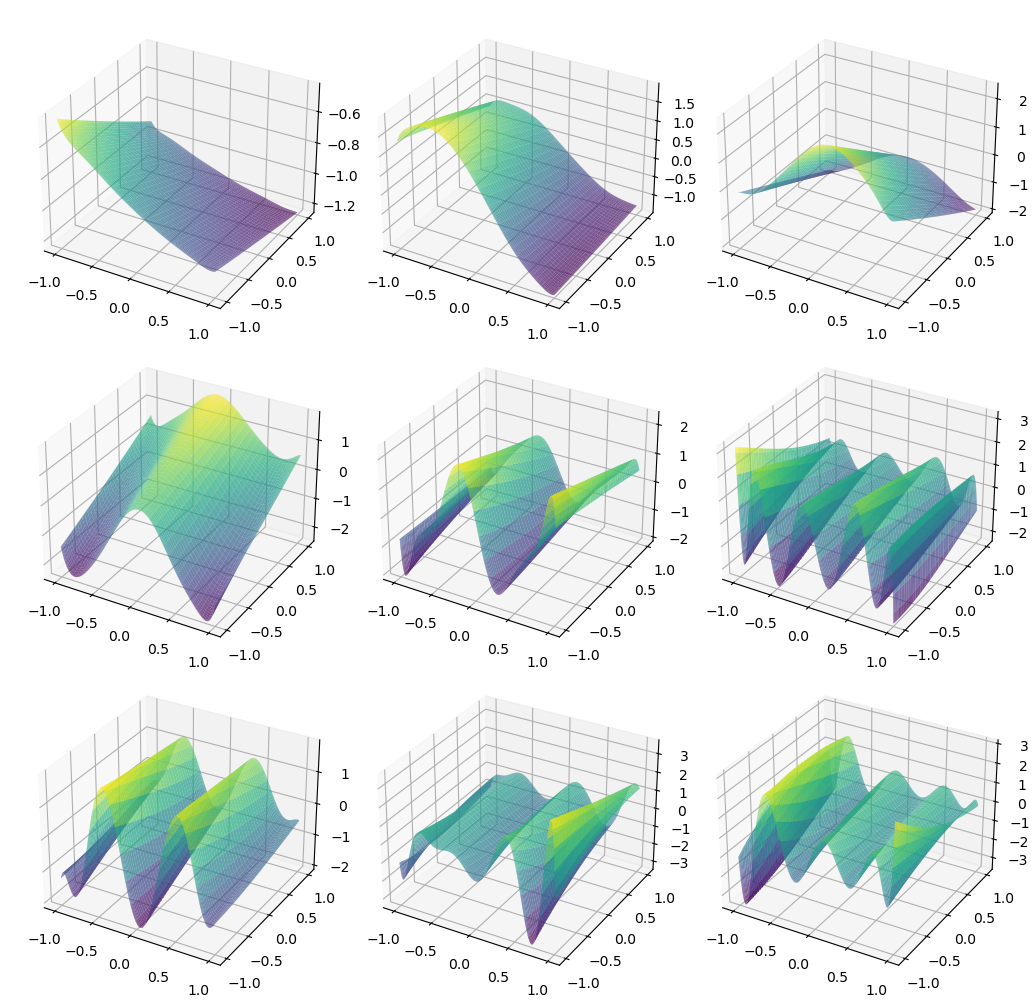}
    \caption{Uncertain basis functions $W_i(t,\bm\xi)$ for $i = 1,\cdots,r$ of the modal DLRA method with the matrix projector-splitting integrator using rank $r = 9$ at time $t_{end}=0.01$.}
    \label{fig:BasisWFig}
\end{figure}
We now turn to studying the approximation quality of DLRA when choosing different ranks in Figure~\ref{fig:L2Errors}. Here, the discrete L$^2$-error of the expectation is depicted in Figure~\ref{fig:L2ErrorsExp}. The inspected methods are SG, where $2^2$, $3^2$ and $4^2$ moments are used, as well as DLRA with matrix projector-splitting and unconventional integrators making use of ranks ranging from $2$ to $16$. Figures~\ref{fig:L2ErrorsExp} and \ref{fig:L2ErrorsVar} depict the L$^2$-error of expectation and variance for the classical methods. Figures~\ref{fig:L2ErrorsExpFilter} and \ref{fig:L2ErrorsVarFilter} depict errors for the filtered methods. First, let us point out that the results indicate a heavily improved error when using the same number of unknowns for DLRA compared to SG. Note however that DLRA requires an increased runtime, since it needs updates of the spatial and uncertain basis functions in addition to updating the coefficient matrix. However, one can state that for the same memory requirement the DLRA method ensures a significantly decreased error for both the expectation and the variance. Comparing the two DLRA integrators, the unconventional integrator leads to an improved approximation of the expectation while the matrix projector-splitting integrator gives an improved variance approximation. The use of filtered DLRA improves the expectation. Furthermore, the filter allows choosing a smaller rank, since the error appears to saturate at a smaller rank. However, the error of the variance increases, which is due to the dampening effect of the variance. Note that the mitigation of spurious oscillations in the variance is not captured by the L$^2$-error. 
\begin{figure}[H]
	\centering
	\begin{subfigure}{0.5\linewidth}
		\centering
		\includegraphics[width=1.0\linewidth]{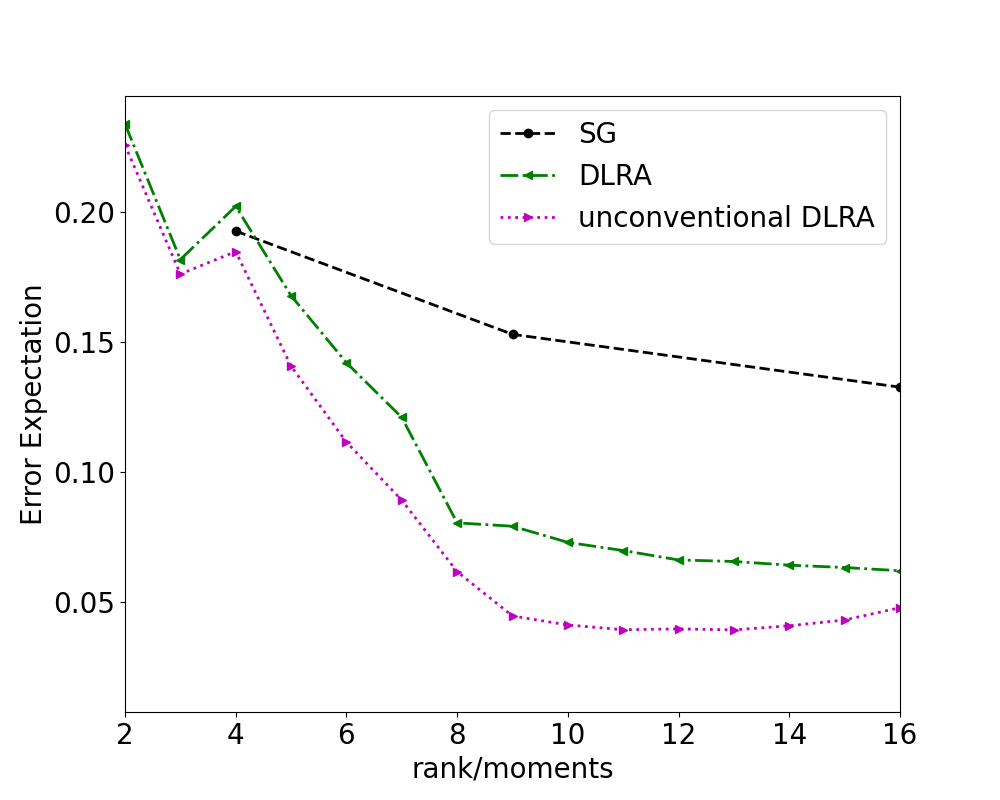}
		\caption{}
		\label{fig:L2ErrorsExp}
	\end{subfigure}%
	\hfill
	\begin{subfigure}{0.5\linewidth}
		\centering
		\includegraphics[width=1.0\linewidth]{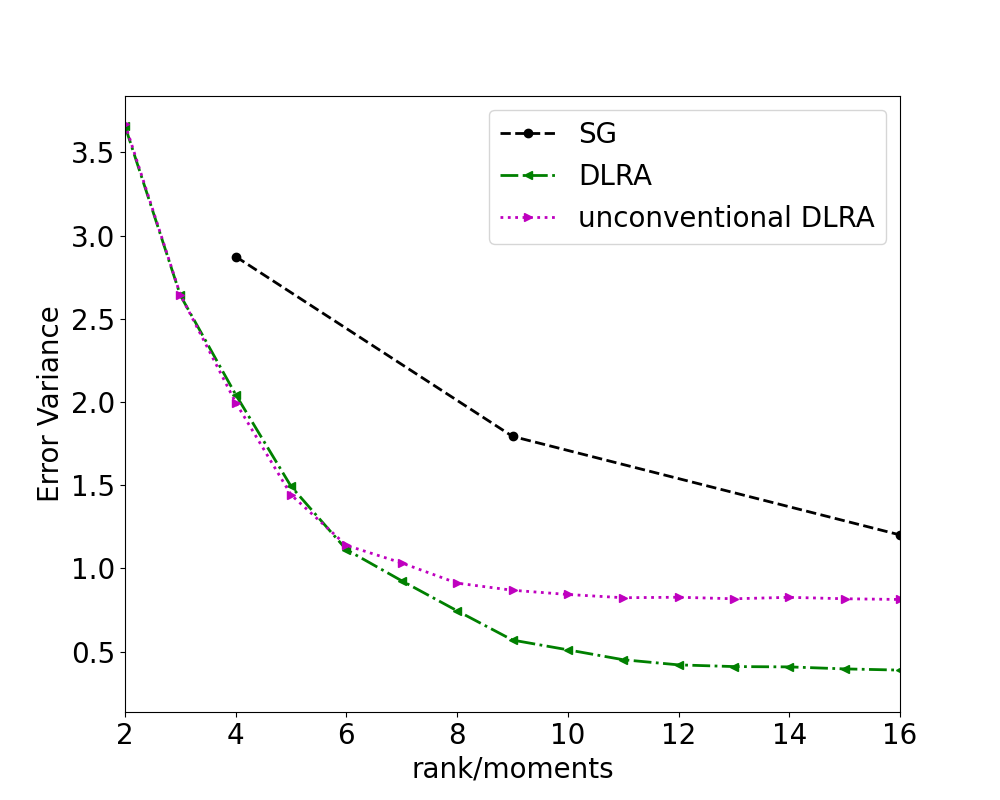}
		\caption{}
		\label{fig:L2ErrorsVar}
	\end{subfigure}\\
	\hfill
	\begin{subfigure}{0.5\linewidth}
		\centering
		\includegraphics[width=1.0\linewidth]{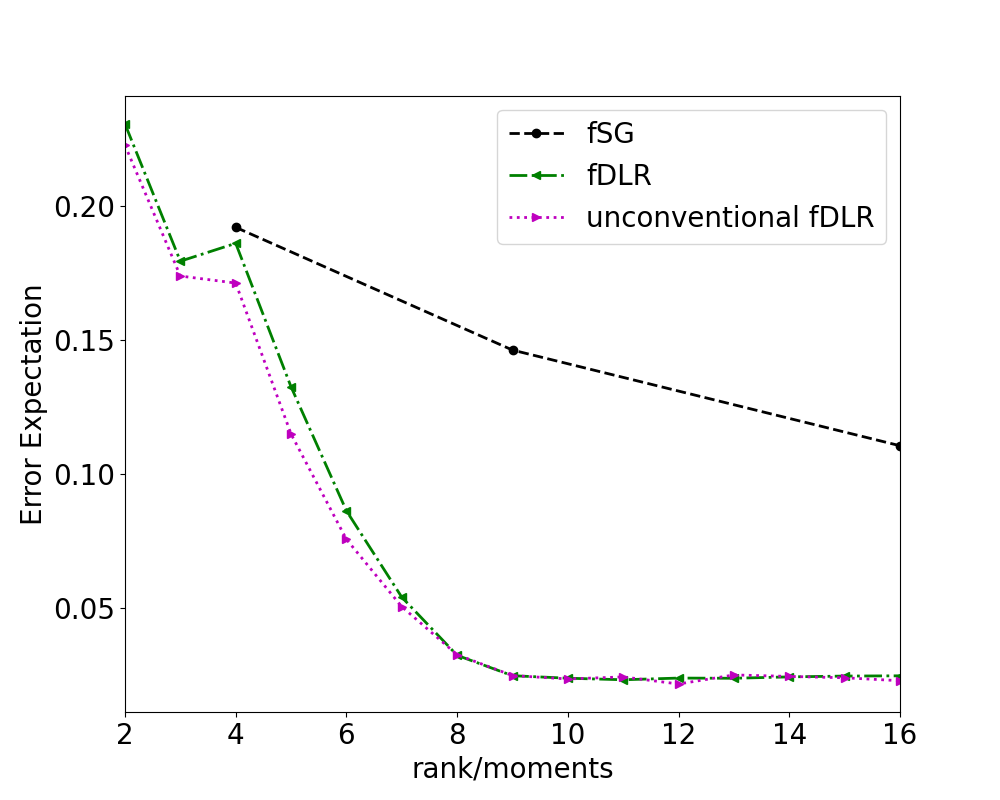}
		\caption{}
		\label{fig:L2ErrorsExpFilter}
	\end{subfigure}%
	\hfill
	\begin{subfigure}{0.5\linewidth}
		\centering
		\includegraphics[width=1.0\linewidth]{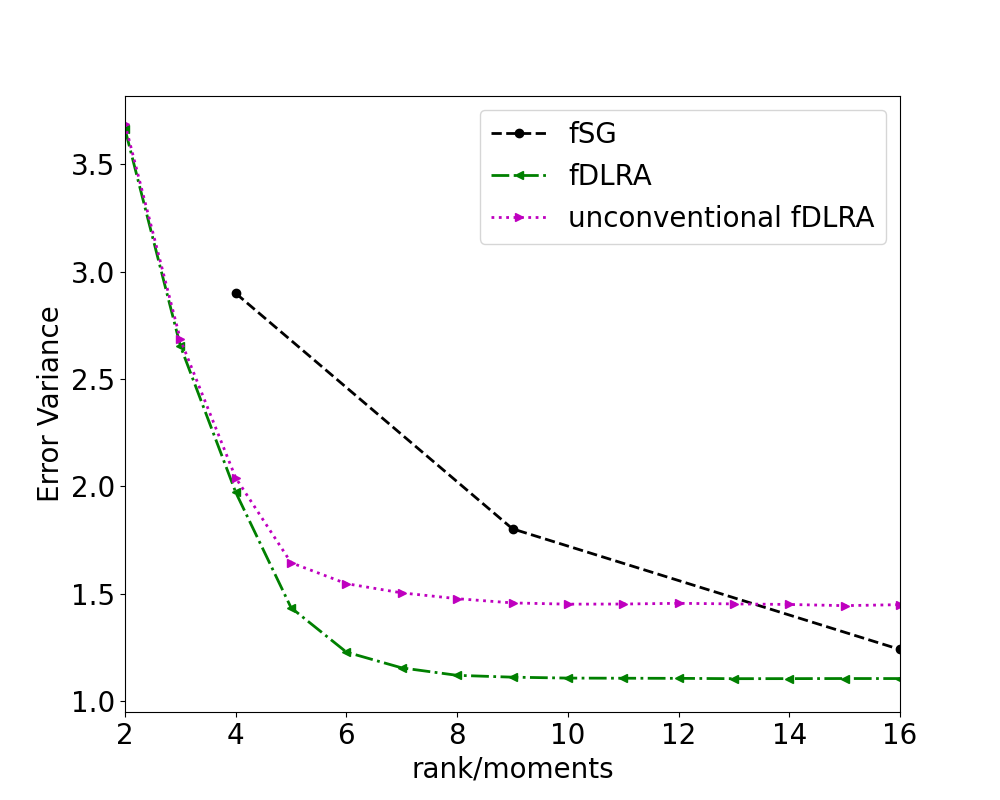}
		\caption{}
		\label{fig:L2ErrorsVarFilter}
	\end{subfigure}
	\hfill
	\caption{L$^2$-error of expectation and variance for DLRA (matrix projector-splitting integrator and unconventional integrator) with varying ranks and SG with varying number of moments. Unfiltered methods are depicted on the top, filtered methods on the bottom.}
	\label{fig:L2Errors}
\end{figure}

\subsection{Summary and Outlook}
In this work, we have derived an efficient representation of the DLRA equations for scalar hyperbolic problems when using the matrix projector-splitting integrator as well as the unconventional integrator. According to the dynamics of the inspected problem, the DLRA method updates the basis functions in time, which allows for an efficient representation of the solution. We have studied modal discretizations and used filters to dampen oscillations. Numerical experiments show a mitigation of the curse of dimensionality through DLRA methods since a reduced number of unknowns is required to represent the uncertainty. A strategy to enforce Dirichlet boundary conditions shows promising results, as boundary conditions are represented exactly while improving the overall solution quality. By applying filters, we can dampen spurious oscillation and thereby ensure a satisfactory result at a lower rank.

In order to further increase the uncertain dimension efficiently, we aim to perform further splitting of the random domain according to \cite{ceruti2020time}. Here, the unconventional integrator will be of high interest, since it allows for parallel solves of all spatial and uncertain basis functions. Further splitting the random domain allows for a significant increase of the number of uncertain dimensions, which we intend to study in future work. Furthermore, we wish to investigate the dampening effects of different DLRA discretizations.


\section*{Acknowledgment} \noindent
The authors would like to thank Christian Lubich and Ryan McClarren for their helpful suggestions and comments. This work was funded by the Deutsche Forschungsgemeinschaft (DFG, German Research Foundation) – Project-ID 258734477 - SFB 1173.

\bibliographystyle{unsrt}  
\bibliography{references}
\end{document}